\documentclass[a4paper, 12pt, reqno,final]{amsart}
\usepackage{amsaddr} %Adresse auf erste Seite

\usepackage[utf8x]{inputenc}
\usepackage[T1]{fontenc}
\usepackage[english]{babel}
\usepackage{amsmath, amsfonts, amsthm, amssymb, amscd}
\usepackage[final]{graphicx}
\usepackage[below]{placeins}
\usepackage{subfig}
\usepackage{multirow}
\usepackage{mathtools}

\usepackage[hidelinks]{hyperref}
\hypersetup{
  pdfauthor = {Stefan Metzger},
  pdftitle = {Nematic flow}
}

\usepackage[final]{showkeys} %HIER AENDERN
\renewcommand\showkeyslabelformat[1]{}
\usepackage{esint}
\usepackage{caption}
\usepackage{dsfont}

\usepackage{tikz}
\usetikzlibrary{patterns}
\usetikzlibrary{shapes.geometric}

\usepackage[]{todonotes}
\renewcommand{\todo}[1]{}
\usepackage{diagbox}

\newtheorem{theorem}{Theorem}[section]
\newtheorem{lemma}[theorem]{Lemma}
\newtheorem*{lemma*}{Lemma}

\newtheorem{remark}[theorem]{Remark}

\numberwithin{equation}{section}

\ifpdf
\hypersetup{
  pdftitle={On a novel approach for modeling liquid crystalline flows},
  pdfauthor={S. Metzger}
}
   \allowdisplaybreaks
\makeatletter
\newcommand{\labitem}[2]{%
\def\@itemlabel{\textbf{#1}}
\item
\def\@currentlabel{#1}\label{#2}}
\makeatother
\newcommand{\norm}[1]{\left\|{#1}\right\|}
\newcommand{\abs}[1]{\left|{#1}\right|}

\newcommand{\rkla}[1]{{\left(#1\right)}}
\newcommand{\trkla}[1]{{(#1)}}
\newcommand{\gkla}[1]{{\left\{#1\right\}}}
\newcommand{\tgkla}[1]{{\{#1\}}}

\newcommand{\ekla}[1]{{\left[#1\right]}}
\newcommand{\tekla}[1]{{[#1]}}
\newcommand{\tabs}[1]{|{#1}|}

\newcommand{\bs}[1]{\boldsymbol{#1}}

\newcommand{\mb}[1]{\mathbf{#1}}

\newcommand{\iOmegaT}{{\int_{\Omega_T}}}

\newcommand{\iOmega}{\int_{\Omega}}

\newcommand{\para}[1]{\partial _{#1}}
\newcommand{\dtau}{\para{\tau}^-}

\renewcommand{\div}{\operatorname{div}}

\newcommand{\nablaX}{\nabla_{\bs{X}}}

\newcommand{\eps}{\varepsilon}

\newcommand{\mol}[1][\empty]{
  \ifthenelse{\equal{#1}{\empty}}
  {\mathcal{J}_{\eps}}
  {\mathcal{J}_{\eps}\gkla{#1}}
}
\newcommand{\molb}[1][\empty]{
  \ifthenelse{\equal{#1}{\empty}}
  {\bs{\mathcal{J}}_{\eps}}
  {\bs{\mathcal{J}}_{\eps}\gkla{#1}}
}
\newcommand{\molbb}[1][\empty]{
  \ifthenelse{\equal{#1}{\empty}}
  {\bs{\mathfrak{J}}_{\eps}}
  {\bs{\mathfrak{J}}_{\eps}\gkla{#1}}
}

\newcommand{\molh}[1][\empty]{
  \ifthenelse{\equal{#1}{\empty}}
  {\mathcal{J}_{h}}
  {\mathcal{J}_{h}\gkla{#1}}
}
\newcommand{\molhb}[1][\empty]{
  \ifthenelse{\equal{#1}{\empty}}
  {\bs{\mathcal{J}}_{h}}
  {\bs{\mathcal{J}}_{h}\gkla{#1}}
}
\newcommand{\molhbb}[1][\empty]{
  \ifthenelse{\equal{#1}{\empty}}
  {\bs{\mathfrak{J}}_{h}}
  {\bs{\mathfrak{J}}_{h}\gkla{#1}}
}

\newcommand{\nn}{^{n}}

\newcommand{\no}{^{n-1}}

\newcommand{\nb}{\nn}

\newcommand{\n}{n}

\newcommand{\tl}{^{\tau}}
\newcommand{\tp}{^{\tau,+}}
\newcommand{\tm}{^{\tau,-}}
\newcommand{\tpm}{^{\tau,(\pm)}}

\newcommand{\uflow}{\mb{u}}

\newcommand{\restr}[2]{\ensuremath{% we make the whole thing an ordinary symbol
  \left.\kern-\nulldelimiterspace % automatically resize the bar with \right
  #1 % the function
  \vphantom{\big|} % pretend it's a little taller at normal size
  \right|_{#2} % this is the delimiter
  }}
  
\newcommand{\xu}{\bs{x}_{\bs{u}}}
\newcommand{\xv}{\bs{x}_{\bs{v}}}
\newcommand{\yu}{\bs{y}_{\bs{u}}}
\newcommand{\yv}{\bs{y}_{\bs{v}}}
\newcommand{\yut}{\tilde{\bs{y}}_{\bs{u}}}
\newcommand{\yvt}{\tilde{\bs{y}}_{\bs{v}}}
\newcommand{\Fu}{\mathds{F}_{\bs{u}}}
\newcommand{\Fv}{\mathds{F}_{\bs{v}}}
\newcommand{\Ev}{\mathds{E}_{\bs{v}}}

\newcommand{\Lag}{\bs{X}}
\newcommand{\iOmegaut}{\int_{\Omega_{\bs{u}}^t}}
\newcommand{\iOmegavt}{\int_{\Omega_{\bs{v}}^t}}
\newcommand{\iOmegazero}{\int_{\Omega_0}}
\newcommand{\varxu}{\delta_{\xu}}
\newcommand{\varxv}{\delta_{\xv}}

\newcommand{\Omegaut}{\Omega_{\bs{u}}^t}
\newcommand{\Omegavt}{\Omega_{\bs{v}}^t}
\newcommand{\Omegazero}{\Omega_0}
\newcommand{\wu}{\bs{w}_{\bs{u}}}
\newcommand{\wv}{\bs{w}_{\bs{v}}}
\newcommand{\varu}{\delta_{\bs{u}}}
\newcommand{\varv}{\delta_{\bs{v}}}
\newcommand{\fconsu}{\bs{f}_{\bs{u}}^{\text{cons}}}
\newcommand{\fconsv}{\bs{f}_{\bs{v}}^{\text{cons}}}
\newcommand{\fdissu}{\bs{f}_{\bs{u}}^{\text{diss}}}
\newcommand{\fdissv}{\bs{f}_{\bs{v}}^{\text{diss}}}

\newcommand{\per}{{\operatorname{per}}}
\newcommand{\Hdivper}{H^1_{\div,\per}\trkla{\Omega}}
\newcommand{\Hdivpermean}{\dot{H}^1_{\div,\per}\trkla{\Omega}}
\begin{document}
 \title[]{On a novel approach for modeling liquid crystalline flows}
\date{\today}
\author[S.~Metzger]{Stefan Metzger}
\address{Department of Applied Mathematics, Illinois Institute of Technology, Chicago IL, 60616, USA}
\email{smetzger2@iit.edu}

% 35Q35 PDE in fluid mechanics
% 76D05 Navier--Stokes
% 76D03 existence
% 76A05 non-Newtonian fluids
%
%

\maketitle

          \begin{abstract}
 In this paper, we derive a new model for the description of liquid crystalline flows.
While microscopic Doi type models suffer from the high dimensionality of the underlying product space, the more macroscopic Ericksen--Leslie type models describe only the long time behavior of the flow and are valid only close to equilibrium.
By applying an energetic variational approach, we derive a new macroscopic model which shall provide an improved description far from equilibrium.
The novelty of our approach lies in the way the energy is minimized. 
Distinguishing between the velocities of particles and fluid allows us to define the energy dissipation not in terms of chemical potentials but in terms of friction induced by the discrepancies in the considered velocities.
We conclude this publication by establishing the existence of weak solutions to the newly derived model.
          \end{abstract}

\section{Introduction}
Liquid crystal is often viewed as an intermediate state between liquid and solid.
It possesses orientational order, but none or only partial positional order.
In this manuscript, we are concerned with the so called nematic state, where the molecules exhibit only orientational order, while floating around freely.

For the description of nematic phases, there are different approaches at hand which differ in accuracy, complexity, and underlying assumptions (see \cite{Emmrich18} for a comprehensive overview).
In the most accurate description, the so called Doi or Doi--Hess model \cite{doi83,doi88,beris94,hess15}, the orientation of liquid crystalline molecules is described by a distribution function $f\trkla{\bs{x},\bs{q},t}$, which describes the probability density that at time $t$ a molecule at point $\bs{x}$ is aligned in direction $\bs{q}\in\mathds{S}$ ($\mathds{S}$ denotes the unit sphere).
The key ingredient of this approach is a Fokker--Planck equation for $f$, which is coupled with a momentum balance equation providing a solenoidal fluid velocity $\bs{u}$. 
In a generic Doi model (cf. \cite{Emmrich18}), this Fokker--Planck equation reads
\begin{align}\label{eq:intro:FP}
\para{t} f + \nabla\cdot\rkla{\bs{u}f} +\mathcal{R}\cdot\rkla{\bs{q}\times\nabla\bs{u}\bs{q}f} = \nabla\cdot\rkla{f\nabla\mu_f} +\mathcal{R}\cdot\rkla{f\mathcal{R}\mu_f}\,.
\end{align}
Here, $\mu_f$ denotes the chemical potential, i.e. the first variation of a suitable free energy with respect to $f$, and $\mathcal{R}:=\bs{q}\times\tfrac{\partial}{\partial\bs{q}}$ denotes the gradient with respect to $\bs{q}$ restricted to the sphere.
Typically, the energy of the system consists of the kinetic energy, an entropic component, and terms describing the effects of alignment.
A suitable energy (see e.g. \cite{e06}) reads e.g.
\begin{align}\label{eq:energy:doihess}
\begin{split}
\mathcal{E}_{\text{DH}}\trkla{\bs{u},f}=&\iOmega\tfrac12\rho\abs{\bs{u}}^2d\bs{x}+\iOmega\int_{\mathds{S}} f\trkla{\bs{x},\bs{q},t}\trkla{\log f\trkla{\bs{x},\bs{q},t}-1}d\bs{q} d\bs{x}\\
&+\iOmega\int_{\mathds{S}}\iOmega\int_{\mathds{S}}\chi\trkla{\bs{x}-\hat{\bs{x}}}\beta\trkla{\bs{q},\hat{\bs{q}}} f\trkla{\hat{\bs{x}},\hat{\bs{q}},t} d\hat{\bs{q}}d\hat{\bs{x}} f\trkla{\bs{x},\bs{q},t} d\bs{q}d\bs{x}\,.
\end{split}
\end{align}
Here, $\rho$ is the constant mass density, $\chi$ denotes a suitable mollifier modeling the range of interaction of the molecules and $\beta$ is an integral kernel describing interaction of two molecules pointing in the directions $\bs{q}$ and $\hat{\bs{q}}$.
A common choice for this interaction kernel introduced by Maier and Saupe \cite{MaierSaupe1958} reads $\beta\trkla{\bs{q},\hat{\bs{q}}}\sim-\trkla{\bs{q}\cdot\hat{\bs{q}}}^2$.
The energy proposed in \eqref{eq:energy:doihess} can also be extended by additional terms to include the influence of external (e.g. magnetic) fields (cf. \cite{Emmrich18} and the references therein).\\
A common simplification consists of studying only the lowest nonvanishing moment of $f$ (cf. \cite{deGennes95}), i.e. the so called $Q$-tensor which is defined as
\begin{align}
\bs{Q}\trkla{\bs{x},t}:=\int_{\mathds{S}} \rkla{\bs{q}\otimes\bs{q}-\tfrac13\mathds{1}} f\trkla{\bs{x},\bs{q},t} d\bs{q}\,.
\end{align}
This purely macroscopic quantity depends only on the spatial coordinate $\bs{x}$.
Under the additional assumption of an uniaxial state, one may describe the nematic phase by a normalized vector $\bs{d}$ called director.
This approach dates back to the publications of Ericksen and Leslie \cite{Ericksen61, Ericksen69, Leslie68, Leslie92}.
\todo{ab hier neu}
They proposed a director model of the form
\begin{align}\label{eq:director:full}
\rho_1\frac{\text{d}^2\bs{d}}{\text{d}t^2}+\bs{\mu}+c_1\rkla{\para{t}\bs{d}+\rkla{\bs{u}\cdot\nabla}\bs{d}+\rkla{\bs{W}\bs{u}}^T\bs{d}}+c_2\rkla{\bs{Du}}\bs{d}+\lambda\bs{d}=0\,.
\end{align}
Thereby, $\bs{W}$ and $\bs{D}$ denote the skew-symmetric and the symmetric part of the gradient.
$\tfrac{\text{d}^2\bs{d}}{\text{d}t^2}$ is the second material derivative, the coefficients $c_1$ and $c_2$ reflect the molecular shape (cf. \cite{Jeffery1922}), and $\lambda$ is the Lagrange multiplier to the length constraint $\abs{\bs{d}}\equiv1$.
The chemical potential $\bs{\mu}$ is again the first variation of a suitable free energy with respect to $\bs{d}$. 
Common choices for the free energy are the Oseen--Franck free energy density (cf. \cite{Franck1958})
\begin{align}\label{eq:energy:oseenfranck}
k_1\trkla{\div\bs{d}}^2+k_2\trkla{\bs{d}\cdot\nabla\times\bs{d}+q}+k_3\abs{\bs{d}\times\nabla\times\bs{d}}^2 +\alpha\trkla{\operatorname{tr}\trkla{\nabla\bs{d}}^2-\trkla{\div\bs{d}}^2}\,,
\end{align}
where $k_i$ ($i = 1, 2, 3$), $\alpha$ and $q$ denote elastic constants, and its one constant approximation $\tfrac12\abs{\nabla\bs{d}}^2$.
Although there are recent analytical results for \eqref{eq:director:full} (cf. \cite{Chechkin2016, Chen2019arxiv}), most results were achieved for a simplified version of \eqref{eq:director:full}.
In particular, the constant $\rho_1$, which is related to the moment of inertia, is assumed to be very small allowing to neglect the first term in \eqref{eq:director:full}.
In addition, the length constraint $\tabs{\bs{d}}\equiv1$ is often approximated by adding a penalty term of the form
\begin{align}
W\trkla{\bs{d}}=\tfrac{1}{4\gamma}\trkla{\abs{\bs{d}}^2-1 }^2\label{eq:def:doublewell}
\end{align}
with $0<\gamma<\!\!\!<1$ to the free energy density.
This approach was introduced by Lin and Liu in \cite{LinLiu95} for a further simplified model, where the director equation reads
\begin{align}
\para{t}\bs{d}+ \rkla{\bs{u}\cdot\nabla}\bs{d}=-\bs{\mu}\,.
\end{align}
Existence results for \eqref{eq:director:full} with $\rho_1=c_2=0$ and the aforementioned approximation of the length constraint can be found in \cite{Lin2000} and in \cite{CavaterraRoccaWu2013} for the case $c_2\neq0$.
For further analytical results like the connection between Parodi's relation and the well-posedness of the system or its long-time behavior can be found in \cite{WuXuLiu2013} and \cite{Petzeltova2013}, respectively.
In the recent years, Ericksen--Leslie type models were also extended to the non-isothermal case (cf.\cite{Feireisl2010,Hieber2018,DeAnnaLiu2019}).
For stochastic Ericksen--Leslie equations, we refer to \cite{Brzezniak2019} and the references therein.
A more extensive survey can be found in \cite{Emmrich18}.\\
The above presented approaches differ not only in their complexity, but also in their theoretical justification.
While the Doi--Hess model arises from molecular kinetic theory, the $Q$-tensor theory and the Ericksen--Leslie theory are phenomenological \cite{Han15} and rely on phenomenological parameters which are often hard to determine from experimental results.
Therefore, the connection of the later theories to the Doi--Hess model was intensely studied in the recent years (cf. \cite{e06,Han15,WangZhangZhang2015,doi83}).
As it turned out, the $Q$-tensor theory and the Ericksen--Leslie theory can be derived from the Doi--Hess description under certain assumptions like that the system is close to equilibrium (cf. \cite{Han15}) or that the Deborah number is vanishing (cf. \cite{e06,WangZhangZhang2015}), i.e. the observation time is significantly larger than the relaxation time.
Both assumptions indicate that this connection between the director models and the Doi--Hess model holds only true when the system is very close to equilibrium.\\
From the mathematical point of view, the Doi--Hess approach and the director description show  major structural differences.
The Fokker--Planck equation appearing in the Doi--Hess description can be written as a generalized balance equation, where the distribution function $f$ is transported by the velocity $\bs{u}-\nabla\mu_f$ (see also the model derivation in \cite{GrunMetzger16}). 
Equations of such structure can be derived using an energetic variational approach (\textit{EnVarA}), which shall be explained in Section \ref{sec:derivation} in more detail.
However, in the typical director equation
\begin{align}\label{eq:intro:director}
\para{t}\bs{d}+\trkla{\bs{u}\cdot\nabla}\bs{d}-\nabla\bs{u}\cdot\bs{d}=-\bs{\mu}\,,
\end{align}
only the left-hand side which describes the deformation induced by the fluid velocity can be obtained via the \textit{EnVarA}.
The right-hand side, which is the negative $L^2$-gradient of the free energy $\iOmega\tfrac12\tabs{\nabla\bs{d}}^2+\iOmega W\trkla{\bs{d}}$, is responsible for the energy minimization, but can not be written as a transport term. 
Consequently, it is not possible to derive \eqref{eq:intro:director} via the \textit{EnVarA}.
Therefore, both models use different mechanisms for the minimization of the free energy. In the Doi--Hess model, the energy enters as an additional contribution to the particle velocity, while the free energy of the director field enters the evolution equation only via an additional $L^2$-gradient.
%Therefore, both models describe a minimization of the free energy, but the applied mechanisms are different.
This might indicate, why assumptions like a long observation time are needed to connect these models.

The aim of this manuscript is to overcome that issue and combine the advantages, i.e. the accuracy of \eqref{eq:intro:FP} and the simplistic description of \eqref{eq:intro:director}.
Therefore, we derive a director model which mimics the behavior of \eqref{eq:intro:FP} in the sense that the chemical potential extends the transport velocity by an additional term.
The main tool for the derivation of the new model is an energetic variational approach (cf. \cite{HyonKwakLiu2010}).
The model describes the evolution of the director field $\bs{d}$ and the velocity $\bs{u}$ in a domain $\Omega\subset\mathds{R}^d$ ($d\in\tgkla{2,3}$) with periodic boundary conditions via
\begin{subequations}\label{eq:model:intro}
\begin{multline}
\para{t}\bs{d}+\rkla{\ekla{\bs{u}+\bs{\mu}\cdot\nabla\bs{d}+\alpha\div\gkla{\bs{\mu}\otimes\bs{d}}-\trkla{1-\alpha} \div\gkla{\bs{d}\otimes\bs{\mu}}}\cdot\nabla}\bs{d}\\
-\alpha\nabla\ekla{\bs{u}+\bs{\mu}\cdot\nabla\bs{d}+\alpha\div\gkla{\bs{\mu}\otimes\bs{d}} -\trkla{1-\alpha} \div\gkla{\bs{d}\otimes\bs{\mu}}}\cdot\bs{d}\\
+\trkla{1-\alpha}\nabla^T\ekla{\bs{u}+\bs{\mu}\cdot\nabla\bs{d}+\alpha\div\gkla{\bs{\mu}\otimes\bs{d}} -\trkla{1-\alpha} \div\gkla{\bs{d}\otimes\bs{\mu}}}\cdot\bs{d}=0\,,\label{eq:model:intro:d}
\end{multline}
\begin{align}
\bs{\mu}=-\Delta\bs{d}+\bs{f}\trkla{\bs{d}}\label{eq:model:intro:mu}\,,
\end{align}
\begin{multline}
\rho\rkla{\para{t}\bs{u}+\rkla{\bs{u}\cdot\nabla}\bs{u}}+\nabla p-\div\gkla{2\eta\mb{D}\bs{u}}\\-\bs{\mu}\cdot\nabla\bs{d}-\alpha\div\gkla{\bs{\mu}\otimes\bs{d}}+\trkla{1-\alpha}\div\gkla{\bs{d}\otimes\bs{\mu}}=0\,,\label{eq:model:intro:u}
\end{multline}
\begin{align}
\div\bs{u}=0\,,\label{eq:model:intro:div}
\end{align}
\end{subequations}
where $\bs{f}$ denotes the variation of $W$ with respect to $\bs{d}$.

The structure of this manuscript is as follows.
In Section \ref{sec:derivation}, we derive the governing equations \eqref{eq:model:intro} by applying an energetic variational approach.
The resulting model is analyzed in Section \ref{sec:existence}. After introducing the used function spaces, we establish the existence of weak solutions to a regularized version of \eqref{eq:model:intro}.
\section{Derivation of the model}
\label{sec:derivation}
In this section, we will make use of an energetic variational approach (\emph{EnVarA}, cf. \cite{HyonKwakLiu2010}) to derive the director model \eqref{eq:model:intro}.
This approach combines the least action principle, which yields an expression for the conservative forces via the variation of the action functional, and the maximum dissipation principle, which allows to determine the dissipative forces.
%Thereby, the least action principle provides the reversible (Hamiltonian) part of the system, while the maximum dissipation principle yields the irreversible (dissipative) part.
The key ingredient for the description of the evolution of the system is the flow map $\xv\trkla{\Lag,t}$, which tracks the position $\xv$ (also referred as Eulerian coordinate) of a material point denoted by the Lagrangian coordinate $\Lag$.
%Here, $\Lag$ is the original labeling (the Lagrangian coordinate) of the particle, which is also referred to as the material coordinate, while $\xv$ is the current (Eulerian) coordinate.
For a given velocity field $\bs{v}\trkla{\xv,t}$, the flow map is defined via
\begin{align}\label{eq:def:flowmap:v}
\para{t}\xv\trkla{\Lag,t}=\bs{v}\trkla{\xv\trkla{\Lag,t},t}\,,\qquad\text{with } \xv\trkla{\Lag,0}=\Lag\,.
\end{align}
This flow map gives rise to the associated deformation tensor $\Fv$ which is given by
\begin{align}\label{eq:def:Fv}
\trkla{\Fv}_{i,j}:=\frac{\partial\trkla{\xv}_i}{\partial\Lag_j}\,.
\end{align}
Without ambiguity, we define $\Fv\trkla{\xv\trkla{\Lag,t},t}=\Fv\trkla{\Lag,t}$.
This allows us to derive the following evolution equations for $ \Fv\trkla{\xv,t}$ and $\Fv^{-T}\trkla{\xv,t}$ (cf. \cite{gurtin1981,larson1998}).
\begin{subequations}
\begin{align}
\para{t}\Fv+\rkla{\bs{v}\cdot\nabla}\Fv&=\nabla\bs{v}\Fv\,,\\
\para{t}\Fv^{-T}+\rkla{\bs{v}\cdot\nabla}\Fv^{-T}&=-\nabla^T\bs{v}\Fv^{-T}\,.
\end{align}
\end{subequations}
The evolution of the director field $\bs{d}$ for the case of general ellipsoid shaped liquid crystal molecules depends on the spatial deformation and can be represented by 
\begin{align}
\bs{d}\trkla{\xv\trkla{\Lag,t},t}=\Ev\trkla{\xv\trkla{\Lag,t},t}\bs{d}_0\trkla{\Lag}
\end{align}
with $\bs{d}_0$ denoting the initial configuration (cf. \cite{Sun09}).
Thereby, the deformation tensor $\Ev$ carries all the information of micro structures and configurations and satisfies the transport equation\todo{Herleitung}
\begin{align}\label{eq:evo:director}
\begin{split}
\para{t}\Ev+\trkla{\bs{v}\cdot\nabla}\Ev&=\ekla{\alpha\nabla\bs{v}-\trkla{1-\alpha}\nabla^T\bs{v}}\Ev\\
&=\mb{W}\bs{v}\Ev+\trkla{2\alpha-1}\mb{D}\bs{v}\Ev\,,
\end{split}
\end{align}
where $\mb{W}$ denotes the skew-symmetric part $\tfrac12\trkla{\nabla-\nabla^T}$ of the gradient and the symbol $\mb{D}:=\tfrac12\trkla{\nabla+\nabla^T}$ denotes its symmetric part (cf. \cite{Jeffery1922}).

As different types of molecules contribute differently to the energy of the system and therefore experience different driving forces, we also allow for different velocity fields.
In particular, we distinguish between the more macroscopic fluid velocity $\bs{u}$, which can be seen as a background flow, and the more microscopic, molecule based particle velocity $\bs{v}$. 
Differences between these velocity fields cause friction which we will include in our model via dissipation (cf. Section \ref{subsec:dissipation}).
This approach was applied successfully i.e. in \cite{XuShengLiu2014} and \cite{LiuWuLiu2018} to derive Poission--Nernst--Planck systems comprising several different ion types.
Adapting these ideas, we allow consider two velocities and consequently two different flow maps.
%As a consequence, we have to use two velocities and two different flow maps.
In addition to the flow map $\xv$ describing the evolution of the directors, we consider a second flow map $\xu$ for the description of the surrounding fluid.
This additional flow map and its associated deformation tensor $\Fu$ are defined analogously to \eqref{eq:def:flowmap:v} and \eqref{eq:def:Fv}, i.e. we have
\begin{align}
\para{t}\xu\trkla{\Lag,t}=\bs{u}\trkla{\xu\trkla{\Lag,t},t}\,,&\qquad\text{with } \xu\trkla{\Lag,0}=\Lag\,,\\
\para{t}\Fu+\rkla{\bs{u}\cdot\nabla}\Fu=\nabla\bs{u}\Fu\,.&
\end{align}
Assuming that the fluid is incompressible, the deformation tensor $\Fu$ additionally satisfies the constraint $\det\Fu=1$ (cf. \cite{Sun09}).
In the following, we will use \emph{EnVarA} to derive \eqref{eq:model:intro} from the basic energy law
\begin{align}\label{eq:energy:basic}
\frac{d\mathcal{E}^{\text{tot}}}{dt}=-\mathcal{D}\,,
\end{align}
where $\mathcal{E}^{\text{tot}}=\mathcal{E}^{\text{kin}}+\mathcal{E}^{\text{int}}$ is the total energy consisting of the kinetic energy and the free (internal) energy of the crystalline molecules.
The change of the total energy is balanced by the total dissipation $\mathcal{D}$.
In conformity with Onsager's linear response assumption, we assume that $\mathcal{D}$ is a linear combination of squares of various rate functions (\cite{HyonKwakLiu2010,Onsager1931a,Onsager1931}).\\
In this manuscript, we investigate a periodic setting, i.e. it suffices to consider only one periodic cell $\Omega\subset\mathds{R}^d$ ($d=2,3$).
Without loss of generality, we may assume that $\Omega$ is a translation of the unit cube $\trkla{0,1}^d$, i.e. $\bs{u}\trkla{\bs{x},t}=\bs{u}\trkla{\bs{x}+k\bs{e}_i,t}$, $\bs{v}\trkla{\bs{x},t}=\bs{v}\trkla{\bs{x}+k\bs{e}_i,t}$, and $\bs{d}\trkla{\bs{x},t}=\bs{d}\trkla{\bs{x}+k\bs{e}_i,t}$ for all $k\in\mathds{Z}$, where the vectors $\bs{e}_1,...,\bs{e}_d$ are the canonical basis of $\mathds{R}^d$.
The reference cell which is used for the Lagrangian description is denoted by $\Omegazero:=\trkla{0,1}^d$.
Furthermore, we introduce $\Omegaut:=\tgkla{\bs{x}\in\mathds{R}^d\,:\, \bs{x}=\xu\trkla{\Lag,t}\text{ with } \Lag\in\Omegazero}$ and  $\Omegavt:=\tgkla{\bs{x}\in\mathds{R}^d\,:\, \bs{x}=\xv\trkla{\Lag,t}\text{ with } \Lag\in\Omegazero}$.
At this point, we want to emphasize that in general $\Omegaut\neq\Omegavt$.
\subsection{Applying the least action principle}
As we consider the energy density as a quantity bound to material points and moving with them, we define the energy on two domains associated with the evolution of $\Omegazero$ despite our periodicity assumption.
In particular, the total energy $\mathcal{E}^{\text{tot}}=\mathcal{E}^{\text{kin}}+\mathcal{E}^{\text{int}}$ is defined via
\begin{align}
\mathcal{E}^{\text{kin}}&=\iOmegaut\tfrac12\rho\abs{\bs{u}\trkla{\xu,t}}^2d\xu\,,\label{eq:kineticenergy}\\
\mathcal{E}^{\text{int}}&=\iOmegavt\tfrac12\abs{\nabla_{\xv}\bs{d}\trkla{\xv,t}}^2d\xv+\iOmegavt W\trkla{\bs{d}\trkla{\xv,t}} d\xv\,,\label{eq:internalenergy}
\end{align}
i.e. the kinetic energy $\mathcal{E}^{\text{kin}}$ with the constant mass density $\rho$ is given in terms of the fluid flow map $\xu$, while the internal energy $\mathcal{E}^{\text{int}}$, which includes the penalty term from \eqref{eq:def:doublewell}, is considered in terms of the particle flow map $\xv$.
The Legendre transformation of \eqref{eq:kineticenergy} and \eqref{eq:internalenergy} yields the action functional $\mathds{A}$ of the particle trajectories in terms of the flow maps $\xu\trkla{\Lag,t}$ and $\xv\trkla{\Lag,t}$:
\begin{align}\label{eq:def:action}
\begin{split}
\mathds{A}\trkla{\xu,\xv}=& \int_0^T\!\!\iOmegaut\tfrac12\rho\abs{\bs{u}\trkla{\xu,t}}^2 d\xu dt -\int_0^T\!\!\iOmegavt\tfrac12\abs{\nabla_{\xv}\bs{d}\trkla{\xv,t}}^2d\xv dt\\
& -\int_0^T\!\!\iOmegavt W\trkla{\bs{d}\trkla{\xv,t}}d\xv dt\\
=&\int_0^T\!\!\iOmegazero\tfrac12\rho\abs{\para{t}\xu\trkla{\Lag,t}}^2d\Lag dt - \int_0^T\!\!\iOmegazero\tfrac12\abs{\Fv^{-T}\nablaX\Ev\bs{d}_0\trkla{\Lag}}^2\det\Fv d\Lag dt\\
&-\int_0^T\!\!\iOmegazero W\trkla{\Ev\bs{d}_0} \det\Fv d\Lag dt\,,
\end{split}
\end{align}
where we used $\det\Fu=1$.
Optimizing this action functional with respect to the trajectories provides the Hamiltonian part of the mechanical system that corresponds to the conservative forces $\fconsu$ and $\fconsv$ (cf. \cite{SonnetVirga2012}).
To compute the variation of $\mathds{A}$ with respect to the flow maps, we consider the one-parameter families of flow maps $\xu^\eps$ and $\xv^\eps$ satisfying
\begin{align}\label{eq:def:xeps}
 \restr{\xu^\eps}{\eps=0}=\xu\,, \quad\restr{\xv^\eps}{\eps=0}=\xv\,, \quad\restr{\frac{d\xu^\eps}{d\eps}}{\eps=0}=\yu\,,\quad\text{and}\quad \restr{\frac{d\xv^\eps}{d\eps}}{\eps=0}=\yv\,,
\end{align}
with a solenoidal $\yu$.
Hence, the variation of the action functional $\mathds{A}$ with respect to $\xu$ yields
\begin{multline}\label{eq:varxu}
\varxu\mathds{A}\trkla{\xu,\xv}=\restr{\frac{d \mathds{A}\trkla{\xu^\eps,\xv}}{d\eps}}{\eps=0}=\int_0^T\iOmegazero\rho\para{t}\xu\cdot\para{t}\yu d\Lag dt\\
=-\int_0^T\iOmegazero\rho\para{tt}\xu\cdot\yu d\Lag dt = -\int_0^T\iOmegaut\rho \rkla{\para{t}\uflow+\trkla{\uflow\cdot\nabla_{\xu}}\uflow}\cdot\yut d\xu dt\,,
\end{multline}
with $\yut\trkla{\xu\trkla{\Lag,t},t}=\yu\trkla{\Lag,t}$.
To accomodate for the constraint $\div_{\xu}\yut$, we add the Lagrange multiplier $\nabla_{\xu}p$.
Therefore, we have
\begin{align}\label{eq:fconsu:1}
\fconsu=-\rho \rkla{\para{t}\uflow+\trkla{\uflow\cdot\nabla_{\xu}}\uflow}-\nabla_{\xu}p
\end{align}
on $\Omegaut$. 
For the variation with respect to $\xv$, we obtain
\begin{align}\label{eq:var:xv:1}
\begin{split}
\varxv\mathds{A}\trkla{\xu,\xv}=&-\int_0^T\!\!\iOmegazero\!\rkla{\Fv^{-T}\nablaX\trkla{\Ev\bs{d}_0\trkla{\Lag}}}\!:\!\ekla{\restr{\frac{d}{d\eps}}{\eps=0}\!\!\rkla{\nabla_{\xv^\eps}\bs{d}\trkla{\xv^\eps,t}}}\det \Fv d\Lag dt\\
&-\int_0^T\!\!\iOmegazero \bs{f}\trkla{\Ev\bs{d}_0\trkla{\Lag}}\cdot\rkla{\restr{\frac{d}{d\eps}}{\eps=0}\Ev^\eps\bs{d}_0\trkla{\Lag}}\det\Fv d\Lag dt\\
&-\int_0^T\!\!\iOmegazero\!\ekla{\tfrac12\abs{\Fv^{-T}\nablaX\trkla{\Ev\bs{d}_0\trkla{\Lag}}}^2\!+\!W\trkla{\Ev\bs{d}_0\trkla{\Lag}}}\restr{\frac{d}{d\eps}}{\eps=0}\!\!\det\Fv^\eps d\Lag dt\\
=:& I+II+III
\end{split}
\end{align}
with $\bs{f}\trkla{\bs{d}}=\tfrac1\gamma\trkla{\tabs{\bs{d}}^2-1}\bs{d}$.
Similar to the computations in \cite{WuXuLiu2013},
we derive for the first integral
\begin{align}\label{eq:var:tmp:1}
\begin{split}
I=&-\int_0^T\!\! \iOmegazero\rkla{\Fv^{-T}\nablaX\trkla{\Ev\bs{d}_0}}:\rkla{\restr{\frac{d\trkla{\Fv^\eps}^{-T}}{d\eps}}{\eps=0}\nablaX\trkla{\Ev\bs{d}_0}}\det\Fv d\Lag dt\\
&-\int_0^T\!\!\iOmegazero\rkla{\Fv^{-T}\nablaX\trkla{\Ev\bs{d}_0}}:\ekla{\Fv^{-T}\nablaX\rkla{\restr{\frac{d\Ev^\eps}{d\eps}}{\eps=0}\bs{d}_0}}\det\Fv d\Lag dt\\
=&-\int_0^T\!\!\iOmegavt\nabla_{\xv}\bs{d} :\rkla{-\nabla_{\xv}^T\yvt\nabla_{\xv}\bs{d}} d\xv dt\\
&-\int_0^T\!\!\iOmegavt\nabla_{\xv}\bs{d}\!:\!\nabla_{\xv}\ekla{\rkla{\tfrac12\trkla{\nabla_{\xv}\yvt\!-\nabla_{\xv}^T\yvt}-\trkla{\tfrac12-\alpha}\trkla{\nabla_{\xv}\yvt\!+\nabla_{\xv}^T\yvt}}\bs{d}}d\xv dt\\
=&-\int_0^T\!\!\iOmegavt\div_{\xv}\gkla{\nabla_{\xv}^T\bs{d}\nabla_{\xv}\bs{d}}\cdot\yvt d\xv dt-\alpha\int_0^T\!\!\iOmegavt\div_{\xv}\gkla{\Delta_{\xv}\bs{d}\otimes\bs{d}}\cdot\yvt d\xv dt\\
&+\trkla{1-\alpha}\int_0^T\!\!\iOmegavt\div_{\xv}\gkla{\bs{d}\otimes\Delta_{\xv}\bs{d}}\cdot\yvt d\xv dt \,%\\
%&+\alpha\int_0^T\idOmegavt\rkla{\rkla{\Delta_{\xv}\bs{d}\otimes\bs{d}}\cdot\bs{n}}\cdot\yvt ds_{\xv}dt\\
%&-\trkla{1-\alpha}\int_0^T\idOmegavt\rkla{\rkla{\bs{d}\otimes\Delta_{\xv}\bs{d}}\cdot\bs{n}}\cdot\yvt ds_{\xv}dt\,
\end{split}
\end{align}
with $\yvt\trkla{\xv\trkla{\Lag,t},t}=\yv\trkla{\Lag,t}$.
\todo{p.I. $\nabla \bs{d}\cdot\bs{n}=0$}
For the second integral on the right-hand side of \eqref{eq:var:xv:1} we compute
\begin{align}\label{eq:var:tmp:2}
\begin{split}
II&=-\int_0^T\!\!\iOmegavt\bs{f}\trkla{\bs{d}}\cdot\ekla{\rkla{\tfrac12\trkla{\nabla_{\xv}\yvt-\nabla_{\xv}^T\yvt}-\trkla{\tfrac12-\alpha}\trkla{\nabla_{\xv}\yvt+\nabla_{\xv}^T\yvt}}\bs{d}}d\xv dt\\
=& \alpha\int_0^T\!\!\iOmegavt \!\div_{\xv}\gkla{\bs{f}\trkla{\bs{d}}\otimes\bs{d}}\cdot\yvt d\xv dt -\trkla{1-\alpha}\int_0^T\!\!\iOmegavt\!\div_{\xv}\gkla{\bs{d}\otimes\bs{f}\trkla{\bs{d}}}\cdot\yvt d\xv dt\,.%\\
%&-\alpha\int_0^T\!\!\idOmegavt\rkla{\rkla{\bs{f}\trkla{\bs{d}}\otimes\bs{d}}\cdot\bs{n}}\cdot\yvt ds_{\xv}dt +\trkla{1-\alpha}\int_0^T\idOmega\rkla{\rkla{\bs{d}\otimes\bs{f}\trkla{\bs{d}}}\cdot\bs{n}}\cdot\yvt ds_{\xv}dt\,.
\end{split}
\end{align}
To deal with the last term, we compute
\begin{align}
\restr{\frac{d}{d\eps}}{\eps=0}\det\Fv^\eps=\det\Fv \operatorname{trace}\gkla{\restr{\frac{d}{d\eps}}{\eps=0}\frac{\partial\xv^\eps}{\partial\Lag}\Fv^{-1}}=\det\Fv \div_{\xv}\gkla{\yvt}\,,
\end{align}
and therefore obtain\todo{$\yvt\cdot\bs{n}$}
\begin{align}\label{eq:var:tmp:3}
\begin{split}
III=&-\int_0^T\iOmegavt\ekla{\tfrac12\abs{\nabla_{\xv}\bs{d}}^2 +W\trkla{\bs{d}}}\div_{\xv}\gkla{\yvt} d\xv dt\\
=&\int_0^T\iOmegavt\rkla{\nabla_{\xv}\bs{d}:\nabla_{\xv}\nabla_{\xv}\bs{d}}\cdot\yvt d\xv dt+\int_0^T\iOmegavt \rkla{\bs{f}\trkla{\bs{d}}\cdot\nabla_{\xv}\bs{d}}\cdot\bs{\yvt}d\xv dt\,.
\end{split}
\end{align}
Combining \eqref{eq:var:tmp:1}, \eqref{eq:var:tmp:2}, and \eqref{eq:var:tmp:3} and introducing $\bs{\mu}:= -\Delta_{\xv}\bs{d}+\bs{f}\trkla{\bs{d}}$, we obtain
\begin{align}
\begin{split}\label{eq:varxv}
\varxv\mathds{A}\trkla{\xu,\xv}=&\int_0^T\iOmegavt\ekla{\nabla_{\xv}\bs{d}:\nabla_{\xv}\nabla_{\xv}\bs{d}+\bs{f}\trkla{\bs{d}}\cdot\nabla_{\xv}\bs{d}}\cdot\yvt d\xv dt\\
&-\int_0^T\iOmegavt\div_{\xv}\gkla{\nabla_{\xv}^T\bs{d}\nabla_{\xv}\bs{d}}\cdot\yvt d\xv dt\\
&+\alpha\int_0^T\iOmegavt\div_{\xv}\gkla{\rkla{-\Delta_{\xv}\bs{d}+\bs{f}\trkla{\bs{d}}}\otimes\bs{d}}\cdot\yvt d\xv dt\\
&-\trkla{1-\alpha}\int_0^T\iOmegavt\div_{\xv}\gkla{\bs{d}\otimes\rkla{-\Delta_{\xv}\bs{d}+\bs{f}\trkla{\bs{d}}}}\cdot\yvt d\xv dt\\
%&-\alpha\int_0^T\idOmegavt\rkla{\rkla{\rkla{-\Delta_{\xv}\bs{d}+\bs{f}\trkla{\bs{d}}}\otimes\bs{d}}\cdot\bs{n}}\cdot\yvt ds_{\xv} dt\\
%& +\trkla{1-\alpha}\int_0^T \idOmegavt\rkla{\rkla{\bs{d}\otimes\rkla{-\Delta_{\xv}\bs{d}+\bs{f}\trkla{\bs{d}}}}\cdot\bs{n}}\cdot\yvt ds_{\xv}dt\\
=&\int_0^T\iOmegavt\rkla{\bs{\mu}\cdot\nabla_{\xv}\bs{d}}\cdot\yvt d\xv dt+\alpha\int_0^T\iOmegavt\div_{\xv}\gkla{\bs{\mu}\otimes\bs{d}}\cdot\yvt d\xv dt\\
&-\trkla{1-\alpha}\int_0^T\iOmegavt\div_{\xv}\gkla{\bs{d}\otimes\bs{\mu}}\cdot\yvt d\xv dt\,.%\\
%&-\alpha\int_0^T\idOmegavt\rkla{\rkla{\bs{\mu}\otimes\bs{d}}\cdot\bs{n}}\cdot\yvt ds_{\xv}dt\\
%&+\trkla{1-\alpha}\int_0^T\idOmegavt\rkla{\rkla{\bs{d}\otimes\bs{\mu}}\cdot\bs{n}}\cdot\yvt ds_{\xv}dt
\end{split}
\end{align}
This yields 
\begin{align}\label{eq:fconsv:1}
\fconsv= \rkla{\bs{\mu}\cdot\nabla_{\xv}\bs{d}} +\alpha \div_{\xv}\gkla{\bs{\mu}\otimes\bs{d}} -\trkla{1-\alpha} \div_{\xv}\gkla{\bs{d}\otimes\bs{\mu}}
\end{align}
on $\Omegavt$.
In \eqref{eq:fconsu:1} and \eqref{eq:fconsv:1}, we stated expressions for the conservative forces $\fconsu$ and $\fconsv$ on $\Omegaut$ and $\Omegavt$, respectively.
The issue of obtaining expressions on different domains is completely owed to the fact that we tracked material points moving with different velocities.
However, due to the periodicity assumption, the derived expressions are valid on each periodicity cell $\Omega$.
As we are no longer interested in the Lagrangian coordinate hiding behind the Eulerian points, we may fix $\Omega=\trkla{0,1}^d$, drop the indices, and obtain
\begin{align}
\fconsu&=-\rho \rkla{\para{t}\uflow+\trkla{\uflow\cdot\nabla}\uflow}-\nabla p\,,\label{eq:def:conservative:fu}\\
\fconsv&= \bs{\mu}\cdot\nabla\bs{d} +\alpha \div\gkla{\bs{\mu}\otimes\bs{d}} -\trkla{1-\alpha} \div\gkla{\bs{d}\otimes\bs{\mu}}\,.\label{eq:def:conservative:fv}
\end{align}
\subsection{Applying the maximum dissipation principle}\label{subsec:dissipation}
Using the maximum dissipation principle \cite{Onsager1931a,Onsager1931}, we perform variations on the dissipation functional with respect to the velocities $\bs{u}$ and $\bs{v}$ to obtain expressions for the dissipative forces.
The dissipation functional is half of the total rate of energy dissipation $\mathcal{D}$ which is defined as
\begin{align}
\mathcal{D}\trkla{\bs{u},\bs{v}}:=2\eta\iOmega \abs{\mb{D}\bs{u}}^2d\bs{x} + \iOmega \abs{\bs{u}-\bs{v}}^2d\bs{x}\,.
\end{align}
Here, the first integral describes the dissipation due to the viscosity of the fluid, while the second integral describes the dissipation due to friction, i.e. due to particle movement relative to the ambient fluid.
Similar to \eqref{eq:def:xeps}, we define $\bs{u}^\eps:=\bs{u}+\eps\wu$, where $\wu$ is an arbitrary regular function with $\div\wu=0$, and $\bs{v}^\eps:=\bs{v}+\eps\wv$.
Then, the variation with respect to the solenoidal fluid velocity $\bs{u}$ provides
\begin{align}
\begin{split}\label{eq:varu}
\varu\rkla{\tfrac12\mathcal{D}\trkla{\bs{u},\bs{v}}}=&\frac12\restr{\frac{d\mathcal{D}\trkla{\bs{u}^\eps,\bs{v}}}{d\eps}}{\eps=0}=2\eta\iOmega\mb{D}\bs{u}:\mb{D}\wu d\bs{x}+\iOmega\rkla{\bs{u}-\bs{v}}\cdot\wu d\bs{x}\,.
\end{split}
\end{align}
The variation with respect to $\bs{v}$ yields
\begin{align}\label{eq:varv}
\begin{split}
\varv\rkla{\tfrac12\mathcal{D}\trkla{\bs{u},\bs{v}}}=&\frac12\restr{\frac{d\mathcal{D}\trkla{\bs{u},\bs{v}^\eps}}{d\eps}}{\eps=0}=\iOmega\rkla{\bs{v}-\bs{u}}\cdot\wv d\bs{x}\,.
\end{split}
\end{align}
Hence, the variation of the dissipation functional with respect to the flow maps provides the negative dissipative forces (cf. \cite{SonnetVirga2012})
\begin{align}
\fdissu&=\div\gkla{2\eta\mb{D}\bs{u}}-\rkla{\bs{u}-\bs{v}}\,,\label{eq:def:dissipative:fu}\\
\fdissv&=-\rkla{\bs{v}-\bs{u}}\,.\label{eq:def:dissipative:fv}
\end{align}
The classical Newton's force balance law states $\bs{f}^{\text{cons}}+\bs{f}^{\text{diss}}=0$.
As this force balance is associated with material points, it depends on the flow map describing their evolution.
Therefore, particles, whose position is given by $\xu$, have to satisfy $\fconsu+\fdissu=0$, while particles, whose position is described by $\xv$, have to satisfy the force balance law $\fconsv+\fdissv=0$ (cf. \cite{XuShengLiu2014, LiuWuLiu2018}).
Consequently, we obtain
\begin{subequations}\label{eq:var:combined}
\begin{align}
\rho\para{t}\bs{u}&+\rho\rkla{\bs{u}\cdot\nabla}\bs{u}+\nabla p=2\eta\div\gkla{\mb{D}\bs{u}}-\rkla{\bs{u}-\bs{v}}\,,\\
\bs{v}=\bs{u}+&\bs{\mu}\cdot\nabla\bs{d}+\alpha\div\gkla{\bs{\mu}\otimes\bs{d}}-\trkla{1-\alpha}\div\gkla{\bs{d}\otimes\bs{\mu}}\,.
\end{align}
\end{subequations}
Combining \eqref{eq:var:combined} with the evolution equation \eqref{eq:evo:director} provides model \eqref{eq:model:intro}.
We want to highlight that our approach using two flow maps differs only in the director equation \eqref{eq:model:intro:d} from the widely-used Ericksen-Leslie model.
The Navier--Stokes equations with the orientation dependent force terms remain the same, as a comparison with the model derived in \cite{WuXuLiu2013} shows.
\begin{remark}\label{rem:lengthconstraint}
In this manuscript, we decided to include the length constraint $\abs{\bs{d}}\equiv1$ via the penalty function $W\trkla{\bs{d}}=\tfrac1{4\gamma}\trkla{1-\abs{\bs{d}}^2}^2$.
Other approaches to include the constraints are the application of Lagrange multipliers or the assumption that the molecules are spherical, which corresponds to choosing $\alpha=0.5$.
In the latter case, the evolution equation \eqref{eq:model:intro:d} simplifies to
\begin{align}\label{eq:d:skewsym}
\para{t}\bs{d}+\trkla{\bs{v}\cdot\nabla}\bs{d}-\bs{W}\bs{v}\bs{d}=0\,.
\end{align} 
When multiplying \eqref{eq:d:skewsym} by $\bs{d}$, the last term including the skew-symmetric gradient of the velocity vanishes and we obtain
\begin{align}
\para{t}\abs{\bs{d}}^2+\trkla{\bs{v}\cdot\nabla}\abs{\bs{d}}^2=0\,,
\end{align}
i.e. when initial length of the director is constant, $\abs{\bs{d}}$ remains constant and no additional constraint on $\abs{\bs{d}}$ is needed.
However, this approach fails for non-spherical molecules.\\
We also refrain from using Lagrange multipliers, as the analytical treatment of the resulting model requires a stronger regularization.
\end{remark}
\section{Existence of weak solutions}\label{sec:existence}
In this section, we will establish the existence of weak solutions to a regularized version of \eqref{eq:model:intro}.
In this version, we add a small portion of the chemical potential weighted with $0<\varepsilon<\!\!<1$ in \eqref{eq:model:intro:d} and obtain
\begin{subequations}\label{eq:model}
\begin{align}
\begin{split}
\para{t}\bs{d}&+\rkla{\ekla{\bs{u}+\bs{\mu}\cdot\nabla\bs{d}+\alpha\div\gkla{\bs{\mu}\otimes\bs{d}}-\trkla{1-\alpha} \div\gkla{\bs{d}\otimes\bs{\mu}}}\cdot\nabla}\bs{d}\\
&-\alpha\nabla\ekla{\bs{u}+\bs{\mu}\cdot\nabla\bs{d}+\alpha\div\gkla{\bs{\mu}\otimes\bs{d}} -\trkla{1-\alpha} \div\gkla{\bs{d}\otimes\bs{\mu}}}\cdot\bs{d}\\
&+\trkla{1-\alpha}\nabla^T\ekla{\bs{u}+\bs{\mu}\cdot\nabla\bs{d}+\alpha\div\gkla{\bs{\mu}\otimes\bs{d}} -\trkla{1-\alpha} \div\gkla{\bs{d}\otimes\bs{\mu}}}\cdot\bs{d}=-\varepsilon\bs{\mu}\,,\label{eq:model:d}
\end{split}
\end{align}
\begin{align}
\bs{\mu}=-\Delta\bs{d}+\bs{f}\trkla{\bs{d}}\label{eq:model:mu}\,,
\end{align}
\begin{align}
\begin{split}
\rho\rkla{\para{t}\bs{u}+\rkla{\bs{u}\cdot\nabla}\bs{u}}+&\nabla p-\div\gkla{2\eta\mb{D}\bs{u}}\\&-\bs{\mu}\cdot\nabla\bs{d}-\alpha\div\gkla{\bs{\mu}\otimes\bs{d}}+\trkla{1-\alpha}\div\gkla{\bs{d}\otimes\bs{\mu}}=0\,,\label{eq:model:u}
\end{split}
\end{align}
\begin{align}
\div\bs{u}=0\label{eq:model:div}\,,
\end{align}
\end{subequations}
in $\Omega$ with periodic boundary conditions.
Similar to the $L^2$-gradient term in \eqref{eq:intro:director}, the regularization $-\varepsilon\bs{\mu}$ can not be derived via the \textit{EnVarA}.
However, in contrast to \eqref{eq:intro:FP}, its influence on the evolution of $\bs{d}$ depends on $\varepsilon$ which can be chosen arbitrarily small.
Throughout this section, we will use the following notation for the function spaces. 
By $W^{k,p}\trkla{\Omega}$, we denote the space of $k$-times weakly differentiable functions with weak derivatives in $L^p\trkla{\Omega}$. 
If we consider only functions with vanishing mean value, we mark the corresponding space by a dot, i.e. $\dot{L}^p\trkla{\Omega}:=\tgkla{f\in L^p\trkla{\Omega}\,:\,\iOmega f=0}$ and $\dot{W}^{k,p}\trkla{\Omega}:=\tgkla{f\in W^{k,p}\trkla{\Omega}\,:\,\iOmega f=0}$.
For $p=2$, we will denote the Hilbert spaces $W^{k,2}\trkla{\Omega}$ and $\dot{W}^{k,2}\trkla{\Omega}$ by $H^k\trkla{\Omega}$ and $\dot{H}^{k}\trkla{\Omega}$, respectively.
Furthermore, we introduce the spaces of (weakly) solenoidal functions given by 
\begin{align}
L^p_{\div}\trkla{\Omega}&:=\gkla{\bs{v}\in L^p\trkla{\Omega}\,:\, \iOmega \bs{v}\cdot\nabla \psi=0\text{ for all }\psi\in W^{1,\tfrac{p}{p-1}}\trkla{\Omega}}\,,\\
\dot{L}^p_{\div}\trkla{\Omega}&:=\gkla{\bs{v}\in L^p_{\div}\trkla{\Omega}\,:\, \iOmega\bs{v}=0}\,,\\
H^1_{\div}\trkla{\Omega}&:= \gkla{\bs{v}\in H^1\trkla{\Omega}\,:\, \div \bs{v}=0}\,,\\
\dot{H}^1_{\div}\trkla{\Omega}&:= \gkla{\bs{v}\in \dot{H}^1\trkla{\Omega}\,:\, \div \bs{v}=0}\,.
\end{align}
To denote the subspaces containing functions which are periodic on $\Omega$, we will use the subscript `$\per$'.\\%To simplify the notation, we introduce $\Hdivper:=H^1_{0,\div,\per}\trkla{\Omega}$.\\
For a Banach space $X$ and a time interval $I$, the symbol $L^p\trkla{I;X}$ stands for the parabolic space of $L^p$-integrable functions on $I$ with values in $X$, while $W^{k,p}\trkla{I;X}$ denotes $k$-times weakly differentiable functions from $I$ to $X$ with weak derivatives in $L^p\trkla{I;X}$.
Furthermore, we use the notation $\dtau f\trkla{t}:=\tau^{-1}\trkla{f\trkla{t}-f\trkla{t-\tau}}$ for backward difference quotients in time of functions in $L^1\trkla{I;X}$.
We will sometimes write $\Omega_T$ for $\Omega\times\rkla{0,T}$.\\
The main result of this section is stated in the following theorem and will be proven in the Sections \ref{subsec:timediscrete} and \ref{subsec:limit}.
\begin{theorem}
\label{th:existence}
Let $\Omega\subset\mathds{R}^d$ ($d\in\tgkla{2,3}$) be a periodic cell.
Then for given initial data $\trkla{\bs{d}^0,\bs{u}^0}\in H^1_{\per}\trkla{\Omega}\times \dot{L}^2_{\div,\per}\trkla{\Omega}$, there exists a triple 
\begin{align}
\bs{d}\in&L^\infty\trkla{0,T;H^1_{\per}\trkla{\Omega}}\cap L^2\trkla{0,T;H^2_{\per}\trkla{\Omega}}\cap W^{1,4/3}\trkla{0,T;\trkla{H^1_{\per}\trkla{\Omega}}^\prime}\,,\\
\bs{\mu}\in&L^2\trkla{0,T;L^2_{\per}\trkla{\Omega}}\,,\\
\bs{u}\in& L^\infty\trkla{0,T;L^2_{\per}\trkla{\Omega}}\cap L^2\trkla{0,T;\Hdivpermean}\cap W^{1,4/3}\trkla{0,T;\trkla{\Hdivper}^\prime}\,
\end{align}
such that $\tekla{\bs{\mu}\cdot\!\nabla\bs{d}+\!\alpha\div\gkla{\bs{\mu}\otimes\bs{d}}-\trkla{1\!-\!\alpha}\div\gkla{\bs{d}\otimes\bs{\mu}}}\!\in \!L^2\trkla{0,T;L^2_\per\trkla{\Omega}}$.
This triple of functions solves \eqref{eq:model} in the following weak sense:
\begin{subequations}\label{eq:result}
\begin{multline}\label{eq:result:d}
\iOmegaT\para{t}\bs{d}\cdot\bs{\theta}+\iOmegaT\ekla{\bs{u}+\bs{\mu}\cdot\nabla\bs{d}+\alpha\div\gkla{\bs{\mu}\otimes\bs{d}}-\trkla{1-\alpha}\div\gkla{\bs{d}\otimes\bs{\mu}}}\cdot\rkla{\bs{\theta}\cdot\nabla\bs{d}}\\
+\iOmegaT\ekla{\bs{u}+\bs{\mu}\cdot\nabla\bs{d}+\alpha\div\gkla{\bs{\mu}\otimes\bs{d}}-\trkla{1-\alpha}\div\gkla{\bs{d}\otimes\bs{\mu}}}:\alpha\div\gkla{\bs{\theta}\otimes\bs{d}}\\
-\iOmegaT\ekla{\bs{u}+\bs{\mu}\nb\cdot\nabla\bs{d}+\alpha\div\gkla{\bs{\mu}\otimes\bs{d}}-\trkla{1-\alpha}\div\gkla{\bs{d}\otimes\bs{\mu}}}:\trkla{1-\alpha}\div\gkla{\bs{d}\otimes\bs{\theta}}
=-\varepsilon\iOmegaT\bs{\mu}\cdot\bs{\theta}\\
\forall\bs{\theta}\in L^4\trkla{0,T; H^1_{\per}\trkla{\Omega}}\,,
\end{multline}\todo{p.I. b.c. $\bs{d}\cdot\bs{n}=0$}
\begin{align}\label{eq:result:mu}
\iOmegaT \bs{\mu}\cdot\bs{\theta}=-\iOmegaT\Delta\bs{d}\cdot\bs{\theta}+\iOmegaT \bs{f}\trkla{\bs{d}}\cdot\bs{\theta}\qquad\qquad%+\iOmegaT\bs{f}_-\trkla{\bs{d}}\cdot\bs{\theta}\\
\forall\bs{\theta}\in L^2\trkla{0,T;L^2_{\per}\trkla{\Omega}}\,,
\end{align}\todo{p.I. b.c. $\nabla\bs{d}\cdot\bs{n}=0$}
\begin{multline}
\rho\iOmegaT\para{t}\bs{u}\cdot\bs{w}+\rho\iOmegaT\rkla{\rkla{\nabla\bs{u}}\cdot\bs{u}}\cdot\bs{w}+\iOmegaT 2\eta\mb{D}\bs{u}:\mb{D}\bs{w}\\
 -\iOmegaT\ekla{\bs{\mu}\cdot\nabla\bs{d}+\alpha\div\gkla{\bs{\mu}\otimes\bs{d}}-\trkla{1-\alpha}\div\gkla{\bs{d}\otimes\bs{\mu}}}\cdot\bs{w} =0\\\
\forall\bs{w}\in L^{4}\trkla{0,T;\Hdivper}\,.\label{eq:result:u}
\end{multline}
\end{subequations}
\end{theorem}
\subsection{Existence of time discrete solution}\label{subsec:timediscrete}
We start by discretizing \eqref{eq:model} in space and time.
Therefore, we subdivide the time interval $I:=[0,T)$ in intervals $I_n:=[t_n,t_{n+1})$ with $t_{n+1}=t_n+\tau$ ($n=0,...,N-1$) for a time increment $\tau=\tfrac{T}{N}>0$ such that $t_N=T$.
To discretize the velocity field, we use normalized eigenfunctions of the Stokes operator with zero mean in a periodic domain, i.e. we consider $\tgkla{\bs{v}_i}_{i=1,...,\infty}$ satisfying
\begin{align}
-\Delta\bs{v}_i+\nabla P_i=\kappa_i\bs{v}_i\,,\quad\text{ with } \div\bs{v}_i=0 \text{~and~} \iOmega\bs{v}_i\text{d}\bs{x}=0\,.
\end{align}
Here, $0<\kappa_1\leq\kappa_2\leq...$ are the eigenvalues of the Stokes operator. 
These eigenfunctions $\bs{v}_i$ are smooth and the sequence $\tgkla{\bs{v}_i}_{i=1,...,\infty}$ forms an orthogonal basis of $\dot{L}^2_{\div,\per}\trkla{\Omega}$ (cf. \cite{Temam95}).
We will denote the corresponding finite dimensional subspace defined by the first $k$ eigenfunctions by $\mathcal{H}_{k,\div}:=\operatorname{span}\tgkla{\bs{v}_1,...,\bs{v}_k}$.
As $H^1_{\per}\trkla{\Omega}$ is separable, we may identify a countable orthonormal basis $\tgkla{\tilde{\bs{\theta}}_k}_{k\in\mathds{N}}$ and define finite dimensional subspaces $\mathcal{H}_k:=\operatorname{span}\tgkla{\tilde{\bs{\theta}}_1,...,\tilde{\bs{\theta}}_k}$ such that $\bigcup_{k\in\mathds{N}}\mathcal{H}_k$ is dense in $H^1_{\per}\trkla{\Omega}$.

Defining the abbreviation 
\begin{align}
\bs{v}_k\nn:=\bs{\mu}_k\nb\cdot\nabla\bs{d}\nn_k+\alpha\div\gkla{\bs{\mu}_k\nb\otimes\bs{d}\nn_k}-\trkla{1-\alpha}\div\gkla{\bs{d}\nn_k\otimes\bs{\mu}_k\nn}\,,
\end{align}
we look for Galerkin solutions $\bs{d}\nn_k\in\mathcal{H}_k$, $\bs{\mu}\nb_k\in\mathcal{H}_k$, and $\bs{u}_k\nn\in\mathcal{H}_{k,\div}$ to
\begin{subequations}\label{eq:discrete}
\begin{multline}\label{eq:discrete:d}
\iOmega\rkla{\bs{d}_k\nn-\bs{d}_k\no}\cdot\bs{\theta}_k
+\tau\iOmega \ekla{\bs{u}_k\nn+\bs{v}_k\nn}\cdot\ekla{\bs{\theta}_k\cdot\nabla\bs{d}_k\nn +\alpha\div\gkla{\bs{\theta}_k\otimes\bs{d}_k\nn}-\trkla{1-\alpha}\div\gkla{\bs{d}_k\nn\otimes\bs{\theta}_k}}\\
=-\varepsilon\tau\iOmega\bs{\mu}_k\nb\cdot\bs{\theta}_k\qquad\forall\bs{\theta}_k\in\mathcal{H}_k\,,
\end{multline}
\begin{align}\label{eq:discrete:mu}
\iOmega \bs{\mu}_k\nb\cdot\bs{\theta}_k=\iOmega\nabla\bs{d}_k\nn:\nabla\bs{\theta}_k+\iOmega \bs{f}_+\trkla{\bs{d}\nn_k}\cdot\bs{\theta}_k+\iOmega\bs{f}_-\trkla{\bs{d}_k\no}\cdot\bs{\theta}_k\qquad\forall\bs{\theta}_k\in\mathcal{H}_k\,,
\end{align}\todo{p.I. b.c. $\nabla\bs{d}\cdot\bs{n}=0$}
\begin{multline}
\rho\iOmega\rkla{\bs{u}_k\nn-\bs{u}_k\no}\cdot\bs{w}_k+\tau\rho\iOmega\rkla{\rkla{\nabla\bs{u}_k\nn}\cdot\bs{u}_k\nn}\cdot\bs{w}_k+\tau\iOmega 2\eta\mb{D}\bs{u}_k\nn:\mb{D}\bs{w}_k\\
 -\tau\iOmega\rkla{\bs{\mu}_k\nb\cdot\nabla\bs{d}_k\nn}\cdot\bs{w}_k-\alpha\tau\iOmega\div\gkla{\bs{\mu}_k\nb\otimes\bs{d}_k\nn}\cdot\bs{w}_k\\
+\trkla{1-\alpha}\tau\iOmega\div\gkla{\bs{d}_k\nb\otimes\bs{\mu}_k\nn}\cdot\bs{w}_k =0\qquad\forall\bs{w}_k\in\mathcal{H}_{k,\div}\,,\label{eq:discrete:u}
\end{multline}
\end{subequations}
where $\bs{d}_k\no\in\mathcal{H}_k$ and $\bs{u}_k\no\in \mathcal{H}_{k,\div}$ are projections of given $\bs{d}^{n-1}\in H^1_{\per}\trkla{\Omega}$ and $\bs{u}^{n-1}\in \dot{L}^2_{\div,\per}\trkla{\Omega}$.
Here, we decomposed $W$ into a convex and a concave part and denoted the corresponding variations by $\bs{f}_+$ and $\bs{f}_-$.
In particular, we use
\begin{subequations}\label{eq:convexconcavesplitting}
\begin{align}
\bs{f}_+\trkla{\bs{d}}&=\tfrac1\gamma\abs{\bs{d}}^2\bs{d}\,,\\
\bs{f}_-\trkla{\bs{d}}&=-\tfrac1\gamma\bs{d}\,.
\end{align}
\end{subequations}
\begin{lemma}[Energy estimate]\label{lem:energy}
A solution $\trkla{\bs{d}_k\nn,~\bs{\mu}_k\nb,~\bs{u}_k\nn}\in \mathcal{H}_{k}\times\mathcal{H}_{k}\times \mathcal{H}_{k,\div}$ to \eqref{eq:discrete}, if it exists, satisfies
\begin{multline}
\iOmega\!\tfrac12\abs{\nabla\bs{d}_k\nn}^2+\iOmega\! W\trkla{\bs{d}_k\nn}+\tau\!\iOmega\!\abs{\bs{\mu}_k\nb\cdot\!\nabla\bs{d}_k\nn+\alpha\div\gkla{\bs{\mu}_k\nb\otimes\bs{d}_k\nn}-\trkla{1\!-\!\alpha}\div\gkla{\bs{d}_k\nn\otimes\bs{\mu}_k\nb}}^2\\
+\varepsilon \tau\iOmega\abs{\bs{\mu}_k\nb}^2 +\iOmega\tfrac12\abs{\nabla\bs{d}_k\nn-\nabla\bs{d}_k\no}^2+\iOmega\tfrac12\abs{\bs{d}_k\nn-\bs{d}_k\no}^2\\
+\iOmega\tfrac12\rho\abs{\bs{u}_k\nn}^2 +\tau\iOmega 2\eta\abs{\mb{D}\bs{u}_k\nn}^2+\tfrac12\iOmega\rho\abs{\bs{u}_k\nn-\bs{u}_k\no}^2\\
\leq\iOmega\tfrac12\abs{\nabla\bs{d}_k\no}^2+\iOmega W\trkla{\bs{d}_k\no} + \iOmega\tfrac12\rho\abs{\bs{u}_k\no}^2\leq C\trkla{\bs{d}\no,\,\bs{u}\no}
\end{multline}
independently of $k$ and $\tau$.
\end{lemma}
\begin{proof}
Testing \eqref{eq:discrete:d} by $\tau\bs{\mu}_k\nb$ and \eqref{eq:discrete:mu} by $\bs{d}_k\nn-\bs{d}_k\no$ yields
\begin{align}
\begin{split}\label{eq:energy:tmp:1}
0\geq &\iOmega\tfrac12\abs{\nabla\bs{d}_k\nn}^2+\iOmega\tfrac12\abs{\nabla\bs{d}_k\nn-\nabla\bs{d}_k\no}^2-\iOmega\tfrac12\abs{\nabla\bs{d}_k\no}^2+\iOmega W\trkla{\bs{d}_k\nn}\\
&-\iOmega W\trkla{\bs{d}_k\no}+\iOmega\tfrac12\abs{\bs{d}_k\nn-\bs{d}_k\no}^2\\
&+\tau\iOmega\bs{u}_k\nn\cdot\ekla{\bs{\mu}_k\nb\cdot\nabla\bs{d}_k\nn+\alpha\div\gkla{\bs{\mu}_k\nb\otimes\bs{d}_k\nn}-\trkla{1-\alpha}\div\gkla{\bs{d}_k\nn\otimes\bs{\mu}_k\nb}}\\
&+\tau\iOmega\abs{\bs{\mu}_k\nb\cdot\nabla\bs{d}_k\nn+\alpha\div\gkla{\bs{\mu}_k\nb\otimes\bs{d}_k\nn}-\trkla{1-\alpha}\div\gkla{\bs{d}_k\nn\otimes\bs{\mu}_k\nb}}^2 +\varepsilon\tau\iOmega\abs{\bs{\mu}_k\nb}^2\,.
\end{split}
\end{align}
Testing \eqref{eq:discrete:u} by $\tau\bs{u}_k\nn$ shows
\begin{align}
\begin{split}\label{eq:energy:tmp:3}
0=&\iOmega\tfrac12\rho\abs{\bs{u}_k\nn}^2+\iOmega\tfrac12\rho\abs{\bs{u}_k\nn-\bs{u}_k\no}^2-\iOmega\tfrac12\rho\abs{\bs{u}_k\no}^2+\tau\iOmega 2\eta\abs{\mb{D}\bs{u}_k\nn}^2\\
&-\tau\iOmega\ekla{\bs{\mu}_k\nb\cdot\nabla\bs{d}_k\nn+\alpha\div\gkla{\bs{\mu}_k\nb\otimes\bs{d}_k\nn}-\trkla{1-\alpha}\div\gkla{\bs{d}_k\nn\otimes\bs{\mu}_k\nb}}\cdot\bs{u}_k\nn\,.
\end{split}
\end{align}
Combining \eqref{eq:energy:tmp:1} and \eqref{eq:energy:tmp:3}, we obtain
\begin{align}
\begin{split}
0\geq &\iOmega\tfrac12\abs{\nabla\bs{d}_k\nn}^2+\iOmega\tfrac12\abs{\nabla\bs{d}_k\nn-\nabla\bs{d}_k\no}^2-\iOmega\tfrac12\abs{\nabla\bs{d}_k\no}^2+\iOmega W\trkla{\bs{d}_k\nn}\\
&-\iOmega W\trkla{\bs{d}_k\no}+\iOmega\tfrac12\abs{\bs{d}_k\nn-\bs{d}_k\no}^2\\
&+\tau\iOmega\abs{\bs{\mu}_k\nb\cdot\nabla\bs{d}_k\nn+\alpha\div\gkla{\bs{\mu}_k\nb\otimes\bs{d}_k\nn}-\trkla{1-\alpha}\div\gkla{\bs{d}_k\nn\otimes\bs{\mu}_k\nb}}^2\\
&+\varepsilon\tau\iOmega\abs{\bs{\mu}_k\nn}^2+\iOmega\tfrac12\rho\abs{\bs{u}_k\nn}^2+\iOmega\tfrac12\rho\abs{\bs{u}_k\nn-\bs{u}_k\no}^2-\iOmega\tfrac12\rho\abs{\bs{u}_k\no}^2\\
&+\tau\iOmega2\eta\abs{\mb{D}\bs{u}_k\nn}^2
\,,
\end{split}
\end{align}
which provides the result.
\end{proof}

\begin{lemma}[Existence of discrete solutions]
Let $\bs{d}_k\no\in\mathcal{H}_k$ and $\bs{u}_k\no\in\mathcal{H}_{k,\div}$ be given.
Then, there exists a solution $\bs{d}_k\nn\in\mathcal{H}_k$, $\bs{\mu}_k\nb\in\mathcal{H}_k$, and $\bs{u}_k\nn\in\mathcal{H}_{k,\div}$ to \eqref{eq:discrete}.
\end{lemma}
\begin{proof}
Noting that for given $\bs{d}_k\nn$ and $\bs{d}_k\no$ the existence of $\bs{\mu}_k\nb$ is clear, we define $\bs{\mu}_k\trkla{\tilde{\bs{d}}_k}$ via
\begin{align}
\iOmega\bs{\mu}_k\trkla{\tilde{\bs{d}}_k}\cdot\bs{\theta}_k=\iOmega\nabla\tilde{\bs{d}}_k:\nabla\bs{\theta}_k+\iOmega \bs{f}_+\trkla{\tilde{\bs{d}}_k}\cdot\bs{\theta}_k+\iOmega\bs{f}_-\trkla{\bs{d}_k\no}\cdot\bs{\theta}_k\,.
\end{align}
With the above definition, the existence of solutions to \eqref{eq:discrete} is equivalent to the existence of a root of a function $\bs{G}\,:\,\mathcal{H}_k\times\mathcal{H}_{k,\div}\rightarrow\mathcal{H}_k\times\mathcal{H}_{k,\div}$ which is defined via $\bs{G}\trkla{\tilde{\bs{d}},\tilde{\bs{u}}}:=\trkla{\bs{G}_{\bs{d}}\trkla{\tilde{\bs{d}},\tilde{\bs{u}}},\bs{G}_{\bs{u}}\trkla{\tilde{\bs{d}},\tilde{\bs{u}}}}^T$ with
\begin{subequations}
\begin{multline}
%\begin{split}
\iOmega\bs{G}_{\bs{d}}\trkla{\tilde{\bs{d}},\tilde{\bs{u}}}\cdot\bs{\theta}_k=\iOmega\trkla{\tilde{\bs{d}}-\bs{d}\no_k}\cdot\bs{\theta}_k +\tau\varepsilon\iOmega\bs{\mu}_k\trkla{\tilde{\bs{d}}}\cdot\bs{\theta}_k\\
+\tau\iOmega\ekla{\tilde{\bs{u}}+\tilde{\bs{v}}\trkla{\tilde{\bs{d}}}}\cdot\ekla{\bs{\theta}_k\cdot\nabla\tilde{\bs{d}} +\alpha\div\gkla{\bs{\theta}_k\otimes\tilde{\bs{d}}} -\trkla{1-\alpha}\div\gkla{\tilde{\bs{d}}\otimes\bs{\theta}_k} }
%\end{split}
\end{multline}
with $\tilde{\bs{v}}\trkla{\tilde{\bs{d}}}:=\bs{\mu}_k\trkla{\tilde{\bs{d}}}\cdot\nabla\tilde{\bs{d}}+\alpha\div\gkla{\bs{\mu}_k\trkla{\tilde{\bs{d}}}\otimes\tilde{\bs{d}}}-\trkla{1-\alpha}\div\gkla{\tilde{\bs{d}}\otimes\bs{\mu}_k\trkla{\tilde{\bs{d}}}}$ and
\begin{multline}
%\begin{split}
\iOmega\bs{G}_{\bs{u}}\trkla{\tilde{\bs{d}},\tilde{\bs{u}}}\cdot\bs{w}=\rho\iOmega\rkla{\tilde{\bs{u}}-\bs{u}_k\no}\cdot\bs{w}_k +\rho\tau\iOmega\rkla{\nabla\tilde{\bs{u}}\cdot\tilde{\bs{u}}}\cdot\bs{w}_k +\tau\iOmega2\eta\mb{D}\tilde{\bs{u}}:\mb{D}\bs{w}_k\\
-\tau\iOmega\ekla{\bs{\mu}_k\trkla{\tilde{\bs{d}}}\cdot\nabla\tilde{\bs{d}}+\alpha\div\gkla{\bs{\mu}_k\trkla{\tilde{\bs{d}}}\otimes\tilde{\bs{d}}}-\trkla{1-\alpha}\gkla{\tilde{\bs{d}}\otimes\bs{\mu}_k\trkla{\tilde{\bs{d}}}} }\cdot\bs{w}_k
%\end{split}
\end{multline}
for all $\bs{\theta}_k\in\mathcal{H}_k$ and $\bs{w}_k\in\mathcal{H}_{k,\div}$.
\end{subequations}
Under the assumption that $\bs{G}$ has no root in $\overline{B}_R:=\gkla{\trkla{\bs{d},\bs{u}}\in\mathcal{H}_k\times\mathcal{H}_{k,\div}\,:\,\norm{\trkla{\bs{d},\bs{u}}}^2_{\mathcal{H}_k\times\mathcal{H}_{k,\div}}:=\norm{\bs{d}}^2_{H^1\trkla{\Omega}}+\norm{\bs{u}}_{L^2\trkla{\Omega}}^2\leq R^2}$, the function $\bs{T}\,:\,\overline{B}_R\rightarrow\partial\overline{B}_R\subset\overline{B}_R$ defined via
\begin{align}
\bs{T}\trkla{\bs{d},\bs{u}}:=-R\frac{\bs{G}\trkla{\bs{d},\bs{u}}}{\norm{\bs{G}\trkla{\bs{d},\bs{u}}}_{\mathcal{H}_k\times \mathcal{H}_{k,\div}}}
\end{align}
is continuous.
According to Brouwer's fixed point theorem, $\bs{T}$ has at least one fixed point $\trkla{\bs{d}^*,\bs{u}^*}$.
Obviously, this fixed point satisfies $\norm{\bs{d}^*}_{H^1\trkla{\Omega}}^2+\norm{\bs{u}^*}^2_{L^2\trkla{\Omega}}=R^2$.
 In the following, we will show that
\begin{align}\label{eq:existence:contradiction}
0<\iOmega\bs{d}^*\cdot\bs{\mu}_k\trkla{\bs{d}^*}+\iOmega\abs{\bs{u}^*}^2<0
\end{align}
for $R$ large enough.
This contradiction shows that $\bs{G}$ has at least one root and therefore \eqref{eq:discrete} has at least one solution.

To prove the second inequality in \eqref{eq:existence:contradiction}, test $\bs{G}_{\bs{d}}\trkla{\bs{d}^*,\bs{u}^*}$ by $\bs{\mu}_k\trkla{\bs{d}^*}$ and $\bs{G}_{\bs{u}}\trkla{\bs{d}^*,\bs{u}^*}$ by $\bs{u}^*$.
Similar to the computations in the proof of Lemma \ref{lem:energy}, we obtain
\begin{align}
\begin{split}
\iOmega\bs{G}_{\bs{d}}&\trkla{\bs{d}^*,\bs{u}^*}\cdot\bs{\mu}_k\trkla{\bs{d}^*} +\iOmega\bs{G}_{\bs{u}}\trkla{\bs{d}^*,\bs{u}^*}\cdot\bs{u}^* \\
\geq &\tfrac12\iOmega\abs{\nabla\bs{d}^*}^2 +\iOmega W\trkla{\bs{d}^*}+\iOmega\rho\tfrac12\abs{\bs{u}^*}^2-C\\
\geq &c\norm{\bs{d}^*}_{H^1\trkla{\Omega}}^2+\tfrac12\norm{\bs{u}^*}_{L^2\trkla{\Omega}}^2-C
\end{split}
\end{align}
with $C>0$ independent of $\bs{d}^*$ and $\bs{u}^*$.
This shows $\iOmega\bs{G}\trkla{\bs{d}^*,\bs{u}^*}\cdot\begin{psmallmatrix}
\bs{\mu}_k\trkla{\bs{d}^*}\\ \bs{u}^*
\end{psmallmatrix}>0$ for $R$ large enough.
Therefore, we have $\iOmega\bs{T}\trkla{\bs{d}^*,\bs{u}^*}\cdot\begin{psmallmatrix}
\bs{\mu}_k\trkla{\bs{d}^*}\\ \bs{u}^*
\end{psmallmatrix}<0$.
Recalling that $\trkla{\bs{d}^*,\bs{u}^*}$ is a fixed point of $\bs{T}$ provides the desired inequality.\\
To establish the first inequality in \eqref{eq:existence:contradiction}, we test $\rkla{\bs{d}^*,\bs{u}^*}$ by $\trkla{\bs{\mu}_k\trkla{\bs{d}^*},\bs{u}^*}$ to obtain
\begin{align}
\begin{split}
\iOmega\bs{d}^*\cdot\bs{\mu}_k\trkla{\bs{d}^*}+\iOmega\abs{\bs{u}^*}^2\geq&\iOmega\abs{\nabla\bs{d}^*}^2+\iOmega W\trkla{\bs{d}^*}-\iOmega W\trkla{\bs{0}}\\
&+\iOmega\bs{d}^*\cdot\rkla{-\bs{f}_-\trkla{\bs{0}}+\bs{f}_-\trkla{\bs{d}_k\no}}+\iOmega\abs{\bs{u}^*}^2\\
\geq &c\norm{\bs{d}^*}_{H^1\trkla{\Omega}}^2+\norm{\bs{u}^*}^2_{L^2\trkla{\Omega}} -C\,.
\end{split}
\end{align}
Therefore, we obtain \eqref{eq:existence:contradiction} for $R$ large enough.
This contradiction provides the existence of solutions $\bs{d}_k\nn\in\mathcal{H}_k$ and $\bs{u}_k\nn\in\mathcal{H}_{k,\div}$.
\end{proof}

\begin{remark}
The technique used in the above lemma to prove the existence of discrete solutions implies no constraints on the size of the time increment $\tau$.
\end{remark}

\begin{lemma}\label{lem:h2}
A solution $\bs{d}_k\nn\in\mathcal{H}_k$ of \eqref{eq:discrete} sastisfies 
\begin{align}\label{eq:d:h2}
\norm{\bs{d}_k\nn}_{H^2\trkla{\Omega}}^2\leq C\trkla{1+\norm{\bs{\mu}\nb_k}_{L^2\trkla{\Omega}}^2+\norm{\bs{d}_k\nn}_{H^1\trkla{\Omega}}^6+\norm{\bs{d}_k\no}^2_{L^2\trkla{\Omega}}}\leq C\trkla{1+\tfrac1\varepsilon}\,.
\end{align}
\end{lemma}

\begin{proof}
From \eqref{eq:discrete:mu}, we obtain
\begin{align}\label{eq:d:laplace}
\iOmega\nabla\bs{d}_k\nn:\nabla\bs{\theta}_k=\iOmega \bs{\mu}_k\nb\cdot\bs{\theta}_k-\iOmega \bs{f}_+\trkla{\bs{d}\nn_k}\cdot\bs{\theta}_k-\iOmega\bs{f}_-\trkla{\bs{d}_k\no}\cdot\bs{\theta}_k\,,
\end{align}
for all $\bs{\theta}_k\in\mathcal{H}_k$.
As the right-hand side of \eqref{eq:d:laplace} is in $L^2_\per\trkla{\Omega}$, standard regularity results provide \eqref{eq:d:h2}.
\end{proof}
\begin{lemma}\label{lem:convergence:space}
Let the triple $\trkla{\bs{\mu}_k\nb,~\bs{d}_k\nn,~\bs{u}_k\nn}$ be a solution to \eqref{eq:discrete}. 
Then, there exists $\bs{\mu}\nb\in L^2_{\per}\trkla{\Omega}$, $\bs{d}\nn\in H^2_{\per}\trkla{\Omega}$, and $\bs{u}\nn\in \Hdivpermean$ and a subsequence -- again denoted by $\trkla{\bs{\mu}_k\nb,~\bs{d}_k\nn,~\bs{u}_k\nn}$ -- such that for $k\rightarrow\infty$
\begin{subequations}\label{eq:convergence:space}
\begin{align}
\bs{d}_k\nn&\rightharpoonup\bs{d}\nn&&\text{in } H^2_\per\trkla{\Omega}\,,\label{eq:convergence:dh2}\\
\bs{d}_k\nn&\rightarrow\bs{d}\nn&&\text{in } W^{1,p}_\per\trkla{\Omega}\text{ for } p<6\label{eq:convergence:dstrong}\,,\\
\bs{\mu}_k\nb&\rightharpoonup\bs{\mu}\nb&&\text{in } L^2_\per\trkla{\Omega}\label{eq:convergence:mul2}\,,\\
\bs{u}_k\nn&\rightharpoonup\bs{u}\nn&&\text{in } H^1_\per\trkla{\Omega}\label{eq:convergence:uj1}\,,\\
\bs{u}_k\nn&\rightarrow\bs{u}\nn && \text{in } L^p_\per\trkla{\Omega}\text{ for } p<6\label{eq:convergence:ustrong}
\end{align}
together with
\begin{multline}
\ekla{\bs{\mu}_k\nb\cdot\nabla\bs{d}_k\nn+\alpha\div\gkla{\bs{\mu}_k\nb\otimes\bs{d}_k\nn} -\trkla{1-\alpha}\div\gkla{\bs{d}_k\nn\otimes\bs{\mu}_k\nb} }\\
\rightharpoonup \ekla{\bs{\mu}\nb\cdot\nabla\bs{d}\nn+\alpha\div\gkla{\bs{\mu}\nb\otimes\bs{d}\nn} -\trkla{1-\alpha}\div\gkla{\bs{d}\nn\otimes\bs{\mu}\nb} }\label{eq:convergence:v}
\end{multline}
\end{subequations}
in $ L^2_\per\trkla{\Omega}$.
\end{lemma}

\begin{proof}
The convergence expressed in \eqref{eq:convergence:dh2}-\eqref{eq:convergence:ustrong} is a direct consequence of the bounds established in Lemma \ref{lem:energy} and Lemma \ref{lem:h2}. 
Concerning \eqref{eq:convergence:v}, the bound in Lemma \ref{lem:energy} provides the weak convergence towards some limit denoted by $\mathcal{V}$ which we have to identify with the right-hand side of \eqref{eq:convergence:v}.
Choosing $\bs{\theta}\in C_{\per}^\infty\trkla{{\Omega}}$, we compute
\begin{multline}
\iOmega\mathcal{V}\cdot\bs{\theta}\leftarrow \iOmega \ekla{\bs{\mu}_k\nb\cdot\nabla\bs{d}_k\nn+\alpha\div\gkla{\bs{\mu}_k\nb\otimes\bs{d}_k\nn}-\trkla{1-\alpha}\div\gkla{\bs{d}_k\nn\otimes\bs{\mu}_k\nb}}\cdot\bs{\theta}\\
=\iOmega \rkla{\bs{\mu}_k\nb\cdot\nabla\bs{d}_k\nn}\cdot\bs{\theta}- \alpha\iOmega\rkla{\bs{\mu}_k\nb\otimes\bs{d}_k\nn}:\nabla\bs{\theta}+\trkla{1-\alpha}\iOmega\rkla{\bs{d}_k\nn\otimes\bs{\mu}_k\nb}:\nabla\bs{\theta}\\
\rightarrow \iOmega \rkla{\bs{\mu}\nb\cdot\nabla\bs{d}\nn}\cdot\bs{\theta}- \alpha\iOmega\rkla{\bs{\mu}\nb\otimes\bs{d}\nn}:\nabla\bs{\theta}+\trkla{1-\alpha}\iOmega\rkla{\bs{d}\nn\otimes\bs{\mu}\nb}:\nabla\bs{\theta}\\
=\iOmega\ekla{\bs{\mu}\nb\cdot\nabla\bs{d}\nn+\alpha\div\gkla{\bs{\mu}\nb\otimes\bs{d}\nn}-\trkla{1-\alpha}\div\gkla{\bs{d}\nn\otimes\bs{\mu}\nb}}\cdot\bs{\theta}\,.
\end{multline}
\end{proof}

These results allow to pass to the limit in \eqref{eq:discrete}.
We again use the abbreviation $\bs{v}\nn:=\bs{\mu}\nb\cdot\nabla\bs{d}\nn+\alpha\div\gkla{\bs{\mu}\nb\otimes\bs{d}\nn}-\trkla{1-\alpha}\div\gkla{\bs{d}\nn\otimes\bs{\mu}\nb}$ for the additional velocity and obtain
\begin{subequations}\label{eq:time}
\begin{multline}\label{eq:time:d}
\iOmega\rkla{\bs{d}\nn-\bs{d}\no}\cdot\bs{\theta}
+\tau\iOmega\ekla{\bs{u}\nn+\bs{v}\nn}\cdot\ekla{\bs{\theta}\cdot\nabla\bs{d}\nn+\alpha\div\gkla{\bs{\theta}\otimes\bs{d}\nn}-\trkla{1-\alpha}\div\gkla{\bs{d}\nn\otimes\bs{\theta}}}
\\
=-\varepsilon\tau\iOmega\bs{\mu}\nb\cdot\bs{\theta}\qquad\forall\bs{\theta}\in H^1_{\per}\trkla{\Omega}\,,
\end{multline}\todo{p.I. b.c. $\bs{d}\cdot\bs{n}=0$}
\begin{align}\label{eq:time:mu}
\iOmega \bs{\mu}\nb\cdot\bs{\theta}=-\iOmega\Delta\bs{d}\nn\cdot\bs{\theta}+\iOmega \bs{f}_+\trkla{\bs{d}\nn}\cdot\bs{\theta}+\iOmega\bs{f}_-\trkla{\bs{d}\no}\cdot\bs{\theta}\qquad\forall\bs{\theta}\in L^2_{\per}\trkla{\Omega}\,,
\end{align}\todo{p.I. b.c. $\nabla\bs{d}\cdot\bs{n}=0$}
\begin{multline}
\rho\iOmega\rkla{\bs{u}\nn-\bs{u}\no}\cdot\bs{w}+\tau\rho\iOmega\rkla{\rkla{\nabla\bs{u}\nn}\cdot\bs{u}\nn}\cdot\bs{w}+\tau\iOmega 2\eta\mb{D}\bs{u}\nn:\mb{D}\bs{w}\\
 -\tau\iOmega\rkla{\bs{\mu}\nb\cdot\nabla\bs{d}\nn}\cdot\bs{w}-\alpha\tau\iOmega\div\gkla{\bs{\mu}\nb\otimes\bs{d}\nn}\cdot\bs{w}+\trkla{1-\alpha}\tau\iOmega\div\gkla{\bs{d}\nn\otimes\bs{\mu}\nb}\cdot\bs{w} =0\\
\qquad\forall\bs{w}\in\Hdivper\,.\label{eq:time:u}
\end{multline}
\end{subequations}
\subsection{Passage to the limit $\tau\searrow0$}\label{subsec:limit}
A summation of the results of Lemma \ref{lem:energy} and Lemma \ref{lem:h2} over all time steps shows the following regularity results on $\Omega\times T$ for solutions to \eqref{eq:time}.
\begin{lemma}\label{lem:bounds}
Let initial data $\bs{d}^0\in H^1_{\per}\trkla{\Omega}$ and $\bs{u}^0\in\dot{L}^2_{\div,\per}\trkla{\Omega}$ be given. Let $\rkla{\bs{d}\nn}_{n=1,...,N}$, $\rkla{\bs{\mu}\nb}_{n=1,...,N}$, and $\rkla{\bs{u}\nn}_{n=1,...,N}$ be solutions to \eqref{eq:time} for $n=1,...,N$.
Then, the following estimates hold true:
\begin{subequations}
\begin{multline}
\max_{n=0,...,N}\iOmega\tfrac12\abs{\nabla\bs{d}^n}^2+\max_{n=0,...,N}\iOmega W\trkla{\bs{d}^n}+\max_{n=0,...,N}\iOmega\tfrac12\rho\abs{\bs{u}^n}^2\\
+\sum_{n=1}^N \iOmega\tfrac12\abs{\nabla\bs{d}\nn-\nabla\bs{d}\no}^2+\sum_{n=1}^N\iOmega\tfrac12\abs{\bs{d}\nn-\bs{d}\no}^2+\sum_{n=1}^N\iOmega\tfrac12\rho\abs{\bs{u}\nn-\bs{u}\no}^2\\
 +\sum_{n=1}^N \tau\iOmega\abs{\bs{\mu}\nb\cdot\nabla\bs{d}\nn +\alpha\div\gkla{\bs{\mu}\nb\otimes\bs{d}\nn}-\trkla{1-\alpha}\div\gkla{\bs{d}\nn\otimes\bs{\mu}\nb}}^2\\
+\varepsilon\sum_{n=1}^N\tau\iOmega\abs{\bs{\mu}\nb}^2 
+\sum_{n=1}^N\tau\iOmega\abs{\mb{D}\bs{u}\nn}^2
\leq \iOmega\tfrac12\abs{\nabla\bs{d}^0}^2 +\iOmega W\trkla{\bs{d}^0} +\iOmega\tfrac12\rho\abs{\bs{u}^0}^2\leq C\,,\label{eq:time:energy:result:a}
\end{multline}
\begin{align}
\sum_{n=1}^N\tau\norm{\bs{d}\nn}_{H^2\trkla{\Omega}}^2\leq C\trkla{1+\tfrac1\varepsilon}\,.\label{eq:time:energy:result:b}
\end{align}
\end{subequations}
\end{lemma}

\begin{proof}
Similar to the proof of Lemma \ref{lem:energy}, we test \eqref{eq:time:d} by $\bs{\mu}\nn$ and \eqref{eq:time:mu} by $\trkla{\bs{d}\nn-\bs{d}\no}$.
Summing over all time-steps provides \eqref{eq:time:energy:result:a}.
\eqref{eq:time:energy:result:b} can be shown similarly to Lemma \ref{lem:h2} by writing \eqref{eq:time:mu} as a Poisson equation for $\bs{d}\nn$.
The estimates in \eqref{eq:time:energy:result:a} show that the corresponding right-hand side is in $L^2\trkla{0,T;L^2_{\per}\trkla{\Omega}}$ which completes the proof.
\end{proof}

\begin{lemma}
Let $\rkla{\bs{d}\nn}_{n=1,...,N}$, $\rkla{\bs{\mu}\nb}_{n=1,...,N}$, and $\rkla{\bs{u}\nn}_{n=1,...,N}$ be solutions to \eqref{eq:time} for $n=1,...,N$ and given initial data $\bs{d}^0\in H_{\per}^1\trkla{\Omega}$ and $\bs{u}^0\in \dot{L}^2_{0,\div}\trkla{\Omega}$. Then
\begin{subequations}
\begin{align}
\tau\sum_{n=1}^N\norm{\dtau\bs{d}\nn}_{\trkla{H^1_\per\trkla{\Omega}}^\prime}^{4/3}&\leq C\trkla{1+\tfrac1\varepsilon}\,,\label{eq:compactness:d}\\
\tau\sum_{n=1}^N\norm{\dtau\bs{u}\nn}_{\trkla{\Hdivper}^\prime}^{4/3}&\leq C\,.\label{eq:compactness:u}
\end{align} 
\end{subequations}
\end{lemma}
\begin{proof}
For $\bs{\theta}\in H^1_\per\trkla{\Omega}$, we have from \eqref{eq:time:d}
\begin{align}
\begin{split}
\abs{\iOmega\dtau\bs{d}\nn\cdot\bs{\theta}}\leq&\abs{\iOmega\ekla{\bs{u}\nn+\bs{v}\nn}\cdot\ekla{\bs{\theta}\cdot\nabla\bs{d}\nn+\alpha\div\gkla{\bs{\theta}\otimes\bs{d}\nn}-\trkla{1-\alpha}\div\gkla{\bs{d}\nn\otimes\bs{\theta}}}}\\
&+\varepsilon\abs{\iOmega\bs{\mu}\nb\cdot\bs{\theta}}\\
=:&I+II\,.
\end{split}
\end{align}
From Hölder's inequality and the Gagliardo--Nirenberg inequality, we obtain
\begin{align}\label{eq:tmp:hprime:1}
\begin{split}
I\leq & \norm{\bs{u}\nn+\bs{v}\nn}_{L^2\trkla{\Omega}}\rkla{\norm{\nabla\bs{d}\nn}_{L^3\trkla{\Omega}}\norm{\bs{\theta}}_{L^6\trkla{\Omega}} + \norm{\nabla\bs{\theta}}_{L^2\trkla{\Omega}}\norm{\bs{d}\nn}_{L^\infty\trkla{\Omega}}}\\
&+ \norm{\bs{u}\nn+\bs{v}\nn}_{L^2\trkla{\Omega}}\norm{\bs{\theta}}_{L^6\trkla{\Omega}}\norm{\nabla\bs{d}\nn}_{L^3\trkla{\Omega}}\\
\leq&C\norm{\bs{u}\nn+\bs{v}\nn}_{L^2\trkla{\Omega}}\!\rkla{\norm{\bs{d}\nn}_{H^2\trkla{\Omega}}^{1/2}\!+\!\norm{\nabla\bs{d}\nn}_{L^4\trkla{\Omega}}\! +\!\norm{\nabla\bs{d}\nn}_{L^3\trkla{\Omega}}\!+\!\norm{\bs{d}\nn}_{H^1\trkla{\Omega}} }\!\norm{\bs{\theta}}_{H^1\trkla{\Omega}}\\
\leq &C\rkla{\norm{\bs{u}\nn}_{L^2\trkla{\Omega}}+\norm{\bs{v}\nn}_{L^2\trkla{\Omega}}}\rkla{1+\norm{\bs{d}\nn}_{H^2\trkla{\Omega}}^{1/2}}\norm{\bs{\theta}}_{H^1\trkla{\Omega}}\,.
\end{split}
\end{align}
Combining \eqref{eq:tmp:hprime:1} with $\tabs{\iOmega\bs{\mu}\nn\cdot\bs{\theta}}\leq \norm{\bs{\mu}\nn}_{L^2\trkla{\Omega}}\norm{\bs{\theta}}_{H^1\trkla{\Omega}}$ and taken the $4/3$ power provides
\begin{align}
\norm{\dtau\bs{d}\nn}_{\trkla{H^1\trkla{\Omega}}^\prime}^{4/3} \leq C\rkla{\norm{\bs{u}\nn}_{L^2\trkla{\Omega}} +\norm{\bs{v}\nn}_{L^2\trkla{\Omega}}}^{4/3}\rkla{1+\norm{\bs{d}\nn}^{1/2}_{H^2\trkla{\Omega}}}^{4/3} +\varepsilon\norm{\bs{\mu}\nn}_{L^2\trkla{\Omega}}^{4/3}\,.
\end{align}
Applying Young's inequality with exponents $3/2$ and $3$ shows
\begin{multline}
\rkla{\norm{\bs{u}\nn}_{L^2\trkla{\Omega}} +\norm{\bs{v}\nn}_{L^2\trkla{\Omega}}}^{4/3}\norm{\bs{d}\nn}_{H^2\trkla{\Omega}}^{2/3}\\\leq C\rkla{\norm{\bs{u}\nn}_{L^2\trkla{\Omega}} +\norm{\bs{v}\nn}_{L^2\trkla{\Omega}}}^2 + C\norm{\bs{d}\nn}_{H^2\trkla{\Omega}}^2\,.
\end{multline}
Therefore, we have
\begin{multline}
%\begin{split}
\norm{\dtau\bs{d}\nn}_{\trkla{H^1\trkla{\Omega}}^\prime}^{4/3} \leq C\norm{\bs{u}\nn}_{H^1\trkla{\Omega}}^2+ C\norm{\bs{d}\nn}_{H^2\trkla{\Omega}}^2+C\varepsilon\norm{\bs{\mu}\nb}_{L^2\trkla{\Omega}}^2+C\\
+ C\norm{\bs{\mu}\nb\cdot\nabla\bs{d}\nn+\alpha\div\gkla{\bs{\mu}\nb\otimes\bs{d}\nn}-\trkla{1-\alpha}\div\gkla{\bs{d}\nn\otimes\bs{\mu}\nb}}_{L^2\trkla{\Omega}}^2 
\,.
%\end{split}
\end{multline}
Multiplying by $\tau$, summing over all time steps and applying the results of Lemma \ref{lem:bounds} provides \eqref{eq:compactness:d}.\\
To obtain the second part, we choose $\bs{w}\in \Hdivper$ and compute
\begin{align}
\begin{split}
\rho\abs{\iOmega\dtau\bs{u}\nn\cdot\bs{w}}\leq& \rho\abs{\iOmega\rkla{\nabla\bs{u}\nn\cdot\bs{u}\nn}\cdot\bs{w}} + 2\eta\abs{\iOmega\mb{D}\bs{u}\nn:\mb{D}\bs{w}} \\
&+\abs{\iOmega\ekla{\bs{\mu}\nb\cdot\nabla\bs{d}\nn+\alpha\div\gkla{\bs{\mu}\nb\otimes\bs{d}\nn}-\trkla{1-\alpha}\div\gkla{\bs{d}\nn\otimes\bs{\mu}\nb}}\cdot\bs{w}}\\
=:& I+II+III\,.
\end{split}
\end{align}
Applying Hölder's inequality and the Gagliardo--Nirenberg inequality together with standard embedding theorems yields
\begin{align}
I\leq& \rho\norm{\nabla\bs{u}\nn}_{L^2\trkla{\Omega}}\norm{\bs{u}\nn}_{L^3\trkla{\Omega}}\norm{\bs{w}}_{L^6\trkla{\Omega}}\leq C\rho \norm{\nabla\bs{u}\nn}_{L^2\trkla{\Omega}}\norm{\bs{w}}_{H^1\trkla{\Omega}}\norm{\bs{u}\nn}^{1/2}_{H^1\trkla{\Omega}}\,,\\
II\leq& C\eta\norm{\bs{u}\nn}_{H^1\trkla{\Omega}}\norm{\bs{w}}_{H^1\trkla{\Omega}}\,,\\
III\leq &\norm{\bs{\mu}\nb\cdot\nabla\bs{d}\nn+\alpha\div\gkla{\bs{\mu}\nb\otimes\bs{d}\nn}-\trkla{1-\alpha}\div\gkla{\bs{d}\nn\otimes\bs{\mu}\nb}}_{L^2\trkla{\Omega}}\norm{\bs{w}}_{H^1\trkla{\Omega}}\,.
\end{align}
From Young's inequality, we obtain
\begin{align}
\begin{split}
\norm{\dtau\bs{u}\nn}_{\trkla{\Hdivper}^\prime}^{4/3}\leq&\, C\norm{\bs{u}\nn}^2_{H^1\trkla{\Omega}}\\
&+\norm{\bs{\mu}\nb\cdot\nabla\bs{d}\nn+\alpha\div\gkla{\bs{\mu}\nb\otimes\bs{d}\nn}-\trkla{1-\alpha}\div\gkla{\bs{d}\nn\otimes\bs{\mu}\nb}}_{L^2\trkla{\Omega}}^2\,.
\end{split}
\end{align}
Again, multiplying by $\tau$, summing over all time steps and applying the results of Lemma \ref{lem:bounds} provides \eqref{eq:compactness:u}.
\end{proof}

With these results, we pass to the limit $\tau\searrow 0$.
For this purpose, we define time-interpolants of time-discrete functions $a\nn$, $\n=0,...,N$, and introduce some time-index-free notation as follows.
\begin{subequations}
\begin{align}
a\tl\trkla{.,t}&:=\tfrac{t-t\no}{\tau}a\nn\trkla{.}+\tfrac{t\nn-t}{\tau} a\no\trkla{.} &&t\in\tekla{t\no,t\nn},n\geq1\,,\\
a\tp\trkla{.,t}&:=a\nn\trkla{.},\ \ a\tm\trkla{.,t}:=a\no\trkla{.}&&t\in(t\no,t\nn], n\geq1\,.
\end{align}
\end{subequations}
We want to point out that the time derivative of $a\tl$ coincides with the difference quotient, i.e.
\begin{align}
\para{t} a\tl= \para{t}\rkla{\tfrac{t-t\no}{\tau}a\nn+\tfrac{t\nn-t}{\tau} a\no} = \tfrac1\tau a\nn-\tfrac1\tau a\no =\dtau a\nn\,.
\end{align}
If a statement is valid for $a\tl$, $a\tp$, and $a\tm$, we use the abbreviation $a\tpm$.
With this notation, our system reads as follows.
\begin{subequations}\label{eq:cont}
\begin{multline}\label{eq:cont:d}
\iOmegaT\para{t}\bs{d}\tl\cdot\bs{\theta}
+ \iOmegaT\ekla{\bs{u}\tp+\bs{v}\tp}\cdot\ekla{\bs{\theta}\cdot\nabla\bs{d}\tp+\alpha\div\gkla{\bs{\theta}\otimes\bs{d}\tp}-\trkla{1-\alpha}\div\gkla{\bs{d}\tp\otimes\bs{\theta}}}\\
=-\varepsilon\iOmegaT\bs{\mu}\tp\cdot\bs{\theta}\\
\forall\bs{\theta}\in L^2\trkla{0,T; H^1\trkla{\Omega}}\,,
\end{multline}\todo{p.I. b.c. $\bs{d}\cdot\bs{n}=0$}
with $\bs{v}\tp:=\bs{\mu}\tp\cdot\nabla\bs{d}\tp+\alpha\div\gkla{\bs{\mu}\tp\otimes\bs{d}\tp}-\trkla{1-\alpha}\div\gkla{\bs{d}\tp\otimes\bs{\mu}\tp}$,
\begin{multline}\label{eq:cont:mu}
\iOmegaT \bs{\mu}\tp\cdot\bs{\theta}=-\iOmegaT\Delta\bs{d}\tp\cdot\bs{\theta}+\iOmegaT \bs{f}_+\trkla{\bs{d}\tp}\cdot\bs{\theta}+\iOmegaT\bs{f}_-\trkla{\bs{d}\tm}\cdot\bs{\theta}\\
\forall\bs{\theta}\in L^2\trkla{0,T;L^2\trkla{\Omega}}\,,
\end{multline}\todo{p.I. b.c. $\nabla\bs{d}\cdot\bs{n}=0$}
\begin{multline}
\rho\iOmegaT\para{t}\bs{u}\tl\cdot\bs{w}+\rho\iOmegaT\rkla{\rkla{\nabla\bs{u}\tp}\cdot\bs{u}\tp}\cdot\bs{w}+\iOmegaT 2\eta\mb{D}\bs{u}\tp:\mb{D}\bs{w}\\
 -\iOmegaT\rkla{\bs{\mu}\tp\cdot\nabla\bs{d}\tp}\cdot\bs{w}-\alpha\iOmegaT\div\gkla{\bs{\mu}\tp\otimes\bs{d}\tp}\cdot\bs{w}\\
+\trkla{1-\alpha}\iOmegaT\div\gkla{\bs{d}\tp\otimes\bs{\mu}\tp}\cdot\bs{w} =0\\\
\forall\bs{w}\in L^{2}\trkla{0,T;\Hdivper}\,.\label{eq:cont:u}
\end{multline}
\end{subequations}
The bounds established above read
\begin{subequations}\label{eq:bounds:cont}
\begin{multline}
%\begin{split}
\norm{\bs{d}\tpm}_{L^\infty\trkla{0,T;H^1\trkla{\Omega}}}
+\varepsilon\norm{\bs{\mu}\tp}_{L^2\trkla{0,T;L^2\trkla{\Omega}}} 
+\norm{\bs{u}\tpm}_{L^\infty\trkla{0,T;L^2\trkla{\Omega}}}\\
+ \norm{\bs{\mu}\tp\cdot\nabla\bs{d}\tp +\alpha\div\gkla{\bs{\mu}\tp\otimes\bs{d}\tp} -\trkla{1-\alpha} \div\gkla{\bs{d}\tp\otimes\bs{\mu}\tp} }_{L^2\trkla{0,T;L^2\trkla{\Omega}}}\\
 +\norm{\bs{u}\tp}_{L^2\trkla{0,T;H^1\trkla{\Omega}}}
+\tau^{-1/2}\norm{\nabla\bs{d}\tp-\nabla\bs{d}\tm}_{L^2\trkla{0,T;L^2\trkla{\Omega}}}\\
+\tau^{-1/2}\norm{\bs{d}\tp-\bs{d}\tm}_{L^2\trkla{0,T;L^2\trkla{\Omega}}}
+\tau^{-1/2}\norm{\bs{u}\tp-\bs{u}\tm}_{L^2\trkla{0,T;L^2\trkla{\Omega}}}
\leq  C\,,
%\end{split}
\end{multline}
\begin{align}
\norm{\bs{d}\tp}_{L^2\trkla{0,T;H^2\trkla{\Omega}}}&\leq C\trkla{1+\tfrac1\varepsilon}\,,\\
\norm{\para{t}\bs{d}\tl}_{L^{4/3}\trkla{0,T;\trkla{H^1_\per\trkla{\Omega}}^\prime}}&\leq C\trkla{1+\tfrac1\varepsilon}\,,\\
\norm{\para{t}\bs{u}\tl}_{L^{4/3}\trkla{0,T;\trkla{\Hdivper}^\prime}}&\leq C\,.
\end{align} 
\end{subequations}
These bounds give rise to the convergence results stated in the following lemma.
\begin{lemma}
There exists a subsequence -- again denoted by $\trkla{\bs{d}\tpm,\bs{\mu}\tp,\bs{u}\tpm}_\tau$ -- and functions
\begin{align}
\bs{d}\in&L^\infty\trkla{0,T;H^1\trkla{\Omega}}\cap L^2\trkla{0,T;H^2\trkla{\Omega}}\cap W^{1,4/3}\trkla{0,T;\trkla{H^1\trkla{\Omega}}^\prime}\,,\\
\bs{\mu}\in&L^2\trkla{0,T;L^2\trkla{\Omega}}\,,\\
\bs{u}\in& L^\infty\trkla{0,T;L^2\trkla{\Omega}}\cap L^2\trkla{0,T;\Hdivpermean}\cap W^{1,4/3}\trkla{0,T;\trkla{\Hdivper}^\prime}\,,
\end{align}
such that for $\tau\searrow0$
\begin{subequations}
\begin{align}
\bs{d}\tpm\stackrel{*}{\rightharpoonup}& \bs{d} &&\text{in } L^\infty\trkla{0,T;H^1_\per\trkla{\Omega}}\,,\\
\bs{d}\tp\rightharpoonup&\bs{d}&&\text{in } L^2\trkla{0,T;H^2_\per\trkla{\Omega}}\,,\\
\bs{d}\tp\rightarrow& \bs{d}&&\text{in } L^p\trkla{0,T;L^s_\per\trkla{\Omega}} \text{ with } p<\infty,~s\in[1,\tfrac{2d}{d-2})\,,\\
\bs{d}\tp\rightarrow&\bs{d}&&\text{in }L^2\trkla{0,T;W^{1,s}_\per\trkla{\Omega}}\text{ with } s\in[1,\tfrac{2d}{d-2})\,,\\
\para{t}\bs{d}\tl\rightharpoonup&\para{t}\bs{d}&&\text{in }L^{4/3}\trkla{0,T;\trkla{H^1_\per\trkla{\Omega}}^\prime}\,,\\
\bs{\mu}\tp\rightharpoonup&\bs{\mu}&&\text{in }L^2\trkla{0,T;L^2_\per\trkla{\Omega}}\,,\\
\bs{u}\tpm\stackrel{*}{\rightharpoonup}&\bs{u}&&\text{in }L^\infty\trkla{0,T;L^2_\per\trkla{\Omega}}\,,\\
\bs{u}\tp\rightharpoonup&\bs{u}&&\text{in } L^2\trkla{0,T;\Hdivpermean}\,,\\
\bs{u}\tp\rightarrow&\bs{u}&&\text{in } L^2\trkla{0,T;L^s_\per\trkla{\Omega}}\text{ with } s\in[1,\tfrac{2d}{d-2})\,,\\
\para{t}\bs{u}\tl\rightharpoonup&\para{t}\bs{u}&&\text{in } L^{4/3}\trkla{0,T;\trkla{\Hdivper}^\prime}\,,
\end{align}
and
\begin{multline}
\ekla{\bs{\mu}\tp\cdot\nabla\bs{d}\tp+\alpha\div\gkla{\bs{\mu}\tp\otimes\bs{d}\tp}-\trkla{1-\alpha}\div\gkla{\bs{d}\tp\otimes\bs{\mu}\tp}}\\
\rightharpoonup \ekla{\bs{\mu}\cdot\nabla\bs{d}+\alpha\div\gkla{\bs{\mu}\otimes\bs{d}}-\trkla{1-\alpha}\div\gkla{\bs{d}\otimes\bs{\mu}}}
\end{multline}
\end{subequations}
in $L^2\trkla{0,T;L^2_\per\trkla{\Omega}}$, where $d\in\tgkla{2,3}$ denotes the spatial dimension.
\end{lemma}
\begin{proof}
The first convergence results are a direct consequence of the bounds \eqref{eq:bounds:cont} and the Aubin-Lions lemma. The last result can be shown following the lines of the proof of \eqref{eq:convergence:v} in Lemma \ref{lem:convergence:space}.
\end{proof}

Using these results, we may pass to the limit $\tau\searrow0$ in \eqref{eq:cont} which completes the proof of Theorem \ref{th:existence}.

\subsection*{Acknowledgment}
This work was supported by the NSF through grant number NSF-DMS 1759536.
\bibliographystyle{amsplain}
%\bibliography{../../complete}

\providecommand{\bysame}{\leavevmode\hbox to3em{\hrulefill}\thinspace}
\providecommand{\MR}{\relax\ifhmode\unskip\space\fi MR }
% \MRhref is called by the amsart/book/proc definition of \MR.
\providecommand{\MRhref}[2]{%
  \href{http://www.ams.org/mathscinet-getitem?mr=#1}{#2}
}
\providecommand{\href}[2]{#2}

\end{document}